\documentclass[a4paper,11pt]{amsart}

\newfont{\cyr}{wncyr10 scaled 1100}

\usepackage[left=2.7cm,right=2.7cm,top=3.5cm,bottom=3cm]{geometry}
\usepackage{amsthm,amssymb,amsmath,amsfonts,mathrsfs,amscd}
\usepackage[latin1]{inputenc}
\usepackage[all]{xy}
\usepackage{latexsym}
\usepackage{longtable}

\theoremstyle{plain}
\newtheorem{theorem}{Theorem}[section]

\newtheorem{lemma}[theorem]{Lemma}
\newtheorem{proposition}[theorem]{Proposition}

\theoremstyle{definition}

\theoremstyle{remark}

\newtheorem{remark}[theorem]{Remark}

\newcommand{\Q}{\mathbb{Q}}
\newcommand{\Z}{\mathbb{Z}}
\newcommand{\F}{\mathbb{F}}

\newcommand{\C}{\mathbb{C}}

\newcommand{\GL}{\operatorname{GL}}

\newcommand{\ord}{{\operatorname{ord}}}
\newfont{\gotip}{eufb10 at 12pt}

\newcommand{\cO}{{\mathcal O}}

\newcommand{\R}{{\mathbb R}}
\newcommand{\M}{{\mathrm{M}}}

\newcommand{\p}{\mathfrak{p}}

\newcommand{\W}{\mathbb W}

\DeclareMathOperator{\Spec}{Spec}

\DeclareMathOperator{\Hom}{Hom} 
\newcommand{\X}{\mathbb X}
\newcommand{\Y}{\mathbb Y}
\newcommand{\U}{\mathbb U}

\newcommand{\fr}{\mathfrak}
\newcommand{\cl }{\mathcal}

\newcommand{\longmono}{\mbox{$\lhook\joinrel\longrightarrow$}}

\newcommand{\longepi}{\mbox{$\relbar\joinrel\twoheadrightarrow$}}

\newcommand{\smallmat}[4]{\bigl(\begin{smallmatrix}#1&#2\\#3&#4\end{smallmatrix}\bigr)}

\newcommand{\D}{\mathbb D}

\include{thebibliography}
\begin{document}

\title[Control theorems for quaternionic Hida families]{A note on control theorems for quaternionic Hida families of modular forms}
\author{Matteo Longo and Stefano Vigni}

\begin{abstract} 
We extend a result of Greenberg and Stevens (\cite{GS}) on the interpolation of modular symbols in Hida families to the context of non-split rational quaternion algebras. Both the definite case and the indefinite case are considered.  
\end{abstract}

\address{Dipartimento di Matematica Pura e Applicata, Universit\`a di Padova, Via Trieste 63, 35121 Padova, Italy}
\email{mlongo@math.unipd.it}
\address{Department of Mathematics, King's College London, Strand, London WC2R 2LS, United Kingdom}
\email{stefano.vigni@kcl.ac.uk}


\maketitle

\section{Introduction}\label{intro}

Fix an integer $N\geq1$ and a prime number $p\geq5$ not dividing $N$. Let $\X$ denote the set of primitive vectors in $\Y:=\Z_p^2$, i.e., the subset of $\Y$ consisting of those elements which are not divisible by $p$. Write $\tilde\D$ for the group of $\Z_p$-valued measures on $\Y$ and $\D$ for the direct summand of $\tilde\D$ consisting of those measures which are supported on $\X$. Then it is possible to introduce an $\M_2(\Z_p)$-action as well as a $\Z_p[\![\Z_p^\times]\!]$-module structure on $\D$ and $\tilde\D$. Define the $\Z_p$-module of $\D$-valued \emph{modular symbols} on $\Gamma_1(N)$ as 
\[ \W:={\rm Symb}_{\Gamma_1(N)}(\D):=\Hom_{\Gamma_1(N)}(\cl D_0,\D)\simeq H^1_{{\rm cpt}}(\Gamma_1(N)\backslash{\mathcal H},\D) \] 
where $\cl D_0$ is the subgroup of degree $0$ divisors on ${\rm Div}\bigl(\mathbb P^1(\Q)\bigr)$ and ${\mathcal H}$ is the complex upper half plane (for the above isomorphism, see \cite[Theorem 4.2]{GS}). Then $\W$ is endowed with a structure of $\Z_p[\![\Z_p^\times]\!]$-algebra as well as with a structure of Hecke module over the $\Z_p[\![\Z_p^\times]\!]$-Hecke algebra $\cl H$ defined in \cite[(1.6)]{GS}.

Let $\Lambda:=\Z_p[\![1+p\Z_p]\!]$ denote the Iwasawa algebra of $1+p\Z_p$ and let $\cl L:={\rm Frac}(\Lambda)$ be its fraction field. Let $\cl R$ denote the \emph{universal ordinary $p$-adic Hecke algebra} of tame level $N$ defined in \cite[Definition 2.4]{GS}; then we have a natural map of $\Lambda$-algebras $h:\cl H\rightarrow\cl R$. Let $\cl X^{\rm arith}$ denote the subset of 
$\cl X:=\Hom^{\rm cont}_{\Z_p}(\tilde{\cl R},\bar\Q_p)$ made up of the arithmetic points defined in \cite[Definition 2.4]{GS}, where $\tilde{\cl R}$ is the normalization of $\cl R$ in $\cl K:=\cl R\otimes_\Lambda\cl L$. 
For any arithmetic point $\kappa\in\cl X^{\rm arith}$ we can consider the localization $\cl R_{(\kappa)}$ of $\cl R$ at $\kappa$. Define $\W_{\cl R_{(\kappa)}}:=\W\otimes_\Lambda{\cl R_{(\kappa)}}$ and denote $h_{(\kappa)}:\cl H\rightarrow\cl R_{(\kappa)}$ the composition of $h$ with the localization map. Let $\W_{(\kappa)}$ denote the $h_{(\kappa)}$-eigenmodule in $\W_{\cl R_{(\kappa)}}$. Finally, noticing that the matrix $\iota:=\smallmat {-1}001$ induces an involution on $\W$, we get a decomposition $\W_{\cl R_{(\kappa)}}=\W_{\cl R_{(\kappa)}}^+\oplus \W_{\cl R_{(\kappa)}}^-$, where $\iota$ acts on $\W_{\cl R_{(\kappa)}}^\epsilon$ as multiplication by $\epsilon$ for $\epsilon=\pm1$. 

With any $\kappa\in\cl X^{\rm arith}$ we can associate an ordinary $p$-stabilized newform $f_\kappa\in S_k(\Gamma_1(Np),\bar\Q_p)$ of tame conductor $N$ (see \cite[Definition 2.5 and Theorem 2.6]{GS}). Let $F_\kappa$ be the (finite) extension of $\Q_p$ generated by the Fourier coefficients of $f_\kappa$. Then we may consider, for any choice of sign $\pm$, the modular 
symbol  
\[ \Phi_{f_\kappa}^\pm\in {\rm Symb}_{\Gamma_1(Np)}\bigl(L_{k-2}(F_\kappa)\bigr):=\Hom_{\Gamma_1(Np)}\bigl(\cl D_0,L_{k-2}(F_\kappa)\bigr). \] 
Here, for any field $F$ and any integer $n\geq 0$, the symbol $L_n(F)$ denotes the $F$-vector space of homogeneous polynomials of degree $n$ endowed with the right action of $\GL_2(F)$ given by $(f|g)(X,Y):=f((X,Y)g^*)$, where $g^*:=\det(g)g^{-1}$ for $g\in\GL_2(F)$. Recall that $\Phi_{f_\kappa}^\pm$ generates the $1$-dimensional $F_\kappa$-subspace $\W_\kappa^\pm$ of 
${\rm Symb}_{\Gamma_1(Np)}(L_{k-2}(F_\kappa))$ on which complex conjugation acts as $\pm1$ and the Hecke algebra acts via the character associated with $f_\kappa$. 

The $\cl R_{(\kappa)}$-modules $\W_{(\kappa)}^\pm$ and $\W_\kappa^\pm$ are connected by a \emph{specialization map} 
\[ \phi_{\kappa,*}:\W_{(\kappa)}^\pm\longrightarrow\W_\kappa^\pm \] 
(see \cite[Definition 5.6]{GS}) deduced from the map $\phi_\kappa:\D\rightarrow L_{k-2}(\bar\Q_p)$ defined by the integration formula 
\[ \mu\longmapsto\int_{\Z_p^\times\times\Z_p}\epsilon(x)(xY-yX)^{k-2}d\mu(x,y), \]
where $\kappa$ has character $\epsilon$ and weight $k$.  The result we are interested in is \cite[Theorem 5.13]{GS}, which can be stated as follows.

\begin{theorem}[Greenberg--Stevens] \label{GS} 
For any $\kappa\in\cl X^{\rm arith}$ and any choice of sign $\epsilon\in\{\pm1\}$ the map $\phi_{\kappa,*}$ induces an isomorphism
\[ \phi_{\kappa,*}:\W^\epsilon_{(\kappa)}\big/P_\kappa\W^\epsilon_{(\kappa)}\overset\simeq\longrightarrow\W^\epsilon_\kappa \]
where $P_\kappa\subset\Z_p[\![\Z_p^\times]\!]$ is the kernel of $\kappa$.
\end{theorem}

It is worth remarking that a generalization of this result to Hilbert modular forms over totally real fields was proved in \cite[Theorem 3.7]{BL}. 

The aim of the present paper is to extend Theorem \ref{GS} to the context of quaternion algebras over $\Q$ (the reader can find a dictionary between classical Hida families and their quaternionic counterparts in \cite[Sections 5 and 6]{LV}). Although Hida in \cite{hida} does not distinguish between the case of definite quaternion algebras and the case of indefinite quaternion algebras, we prefer to keep these two settings separate. The reason for doing so is that the natural substitutes for $\W$ look quite different in the two cases, and some arguments in the definite case are simpler than the corresponding ones in the indefinite case. The price for this choice is that similar arguments are repeated twice, while the advantage is that the exposition becomes clearer and one can read each of the two parts independently. 

The main result that we obtain can be described as follows. Let $B$ be a quaternion algebra over $\Q$ of discriminant the square-free integer $D>1$. Fix an Eichler order $R$ of $B$ of level $M$ prime to $D$ and let $p$ be a prime not dividing $MD$. Fix also an ordinary $p$-stabilized eigenform $f$ of level $\Gamma_1(MDp)$ and weight $k$, and write $F_f$ for the field generated over $\Q_p$ by its Fourier coefficients, whose ring of integers will be denoted $\mathcal O_f$. For simplicity, we assume that the $p$-adic representation attached to $f$ is residually absolutely irreducible and $p$-distinguished. For every prime $\ell|M$ choose an isomorphism $i_\ell:B\otimes_\Q\Q_\ell\simeq\M_2(\Q_\ell)$ such that $i_\ell(R\otimes_\Z\Z_\ell)$ is the subgroup of upper triangular matrices modulo $\ell^{{\rm ord}_\ell(M)}$. Moreover, choose $i_p: B\otimes_\Q\Q_p\simeq\M_2(\Q_p)$ such that $i_p(R\otimes_\Z\Z_p)=\M_2(\Z_p)$. Define  
\[ \W:=\begin{cases}H^1(\Gamma_0,\D)&\text{if $B$ is indefinite},\\[2mm]S_2(U_0,\D)&\text{if $B$ is definite},\end{cases} \] 
where the notations are as follows:
\begin{itemize} 
\item $\D$ is the $\mathcal O_f$-module of $\mathcal O_f$-valued measures on $\Y$ which are supported on $\X$; 
\item $\Gamma_0$ is a finite index subgroup of the group $R^\times_1$ of norm $1$ elements in $R^\times$, containing the subgroup 
of $R^\times_1$ consisting of the elements $\gamma$ such that $i_\ell(\gamma)\equiv\smallmat 1{*}01\pmod{\ell^{{\rm ord}_\ell(M)}}$ for all primes $\ell|M$; 
\item $U_0$ is a finite index subgroup of $\hat R^\times$ containing the subgroup of $\hat R^\times$ consisting of the elements $u=(u_\ell)_\ell$ such that $i_\ell(u_\ell)\equiv \smallmat 1{*}01\pmod{\ell^{{\rm ord}_\ell(M)}}$ for all primes $\ell|M$;
\item $S_2(U_0,\D)$ is the $\mathcal O_f$-module of $\D$-valued modular forms of weight $2$ and level $U_0$ on $\hat B^\times$ (see \S \ref{definite-forms-subsec}).  
\end{itemize}  
To state our main result, we introduce the following notations, which slightly differ from those used before. Let $\mathcal R$ denote the integral closure of $\Lambda$ in the primitive component $\mathcal K$ of ${\mathfrak h_\infty^{D,\ord}}\otimes_\Lambda\mathcal L$ corresponding to $f$, where now $\mathfrak h_\infty^\ord$ is the $p$-ordinary Hecke algebra of level $\Gamma_0$ (indefinite case) or $U_0$ (definite case) with coefficients in $\mathcal O_f$ associated with $B$ (see \S \ref{App1} and \S \ref{App1-def} for the relevant definitions). Let $\mathcal A(\mathcal R)$ denote the set of arithmetic homomorphisms in $\Hom(\mathcal R,\bar\Q_p)$ (this notion is introduced in \S \ref{sec-arith-points}). A point $\kappa\in\mathcal A(\mathcal R)$ corresponds to a normalized eigenform $f_\kappa$; write $F_\kappa$ for the field generated over $\Q_p$ by the Fourier coefficients of $f_\kappa$. For any $\kappa\in\mathcal A(\mathcal R)$ define 
\[ \W_\kappa:=\begin{cases}H^1\bigl(\Gamma_r,V_{k_\kappa-2}(F_\kappa)\bigr)^{f_\kappa}&\text{if $B$ is indefinite},\\[2mm]
S_2(U_r,F_\kappa)^{f_\kappa}&\text{if $B$ is definite},\end{cases} \] 
where
\begin{itemize} 
\item the superscript $f_\kappa$ denotes the subspace on which the Hecke algebra acts via the character associated with $f_\kappa$;
\item $\Gamma_r$ is the subgroup of $\Gamma_0$ consisting of the elements $\gamma$ such that $i_p(\gamma)\equiv\smallmat 1{*}01\pmod{p^r}$; 
\item $U_r\subset U_0$ is the compact open subgroup of $\hat B^\times$ whose $p$-component is isomorphic via $i_p$ to the group of matrices in $\GL_2(\Z_p)$ congruent to $\smallmat 1{*}0{*}$ modulo $p^r$; 
\item $S_2(U_r,F_\kappa)$ is the $F_\kappa$-vector space of modular forms of weight $2$ and level $U_r$.  
\end{itemize} 
For any field $F$ we may define \emph{specialization maps}  
\[ \rho_{k-2,\epsilon}:\D\longrightarrow V_{k-2}(F) \]
by the formulas
\[ \rho_{k-2,\epsilon}(\nu)(P):=\begin{cases}\int_{\Z_p\times\Z_p^\times}\epsilon(y)P(x,y)d\nu&\text{if $B$ is indefinite},\\[2mm]
\int_{\Z_p^\times\times p\Z_p}\epsilon(x)P(x,y)d\nu&\text{if $B$ is definite}.\end{cases} \] 
For any $\kappa\in\mathcal A(\mathcal R)$ of weight $k_\kappa$ and character $\epsilon_\kappa$ (see \S \ref{sec-arith-points} for 
definitions) we may consider the map $\rho_{k_\kappa-2,\epsilon_\kappa}$ which induces maps: 
\[\rho_\kappa:\W^\ord\longrightarrow \W^\ord_{\kappa}.\] 
Here $\W^\ord$ and $\W^\ord_\kappa$ denote the \emph{ordinary submodules} of $\W$ and $\W_\kappa$, respectively, defined as in \cite[Definition 2.2]{GS} (see also \S \ref{sec3.5} and \S \ref{W-def}). Finally, in this case too there is a universal Hecke algebra ${\mathfrak h^D_{\rm univ}}$ equipped with a canonical morphism $h:{\mathfrak h^D_{\rm univ}}\rightarrow\mathfrak h_\infty^\ord$. For any $\kappa\in\mathcal A(\mathcal R)$ let $P_\kappa$ denote its kernel and $\mathcal R_{P_\kappa}$ the localization of $\mathcal R$ at $P_\kappa$ (note the slight notational change with respect to \cite{GS}). Let $\W^\ord_{h_\kappa}$ be the $h_{\kappa}$-submodule in $\W^\ord$, where $h_{\kappa}:{\mathfrak h^D_{\rm univ}}\rightarrow\mathcal R_{P_\kappa}$ is the composition of $h$ with the localization map $\mathcal R\rightarrow\mathcal R_{P_\kappa}$. 

\begin{theorem} \label{main-intro-thm}
For any $\kappa\in\mathcal A(\mathcal R)$ the specialization map defines an isomorphism
\[ \W_{h_\kappa}^\ord\big/P_\kappa\W^\ord_{h_\kappa}\overset{\simeq}{\longrightarrow} \W_\kappa^\ord. \] 
\end{theorem} 

Related results in the context of Coleman families are available: see the article \cite{Chenevier} by Chenevier (definite case) and the paper \cite{AS-new} by Ash and Stevens. However, in this note we avoid using locally analytic distributions because we work in the simpler setting of ordinary deformations, where we can offer a more explicit and detailed version of this result. 
In fact, this was one of our motivations for writing this paper. 

The above result is a combination of Theorem \ref{control-thm} (indefinite case) and Theorem \ref{control-thm-def} (definite case) and was crucially applied in \cite{LV-darmon} to obtain rationality results for quaternionic Darmon points on elliptic curves. In the indefinite case a more precise version of Theorem \ref{main-intro-thm} can be stated, taking into account the action of the archimedean involution as in Theorem \ref{GS}. We also observe that in the definite case the above result generalizes \cite[Theorem 2.5]{BD} and, actually, provides a full proof of it (in fact, a proof was only briefly sketched in \cite{BD}).

We caution the reader that some of the notations adopted in the main body of the paper may slightly differ from those used in this introduction. For example, as noticed above, in the sequel we use the symbol $\mathcal R$ to denote a single component of the universal ordinary Hecke algebra appearing in Theorem \ref{GS}. Furthermore, localizations of $\mathcal R$ at arithmetic points $\kappa$ are denoted $\mathcal R_{P_\kappa}$ instead of $\mathcal R_{(\kappa)}$ (the latter being the symbol used in \cite{GS}). However, every piece of notation will be carefully defined, and we are confident that no confusion will arise. 
\vskip 2mm
\noindent\emph{Convention.} Throughout the paper we fix field embeddings $\bar\Q\hookrightarrow\bar\Q_p$ and $\bar\Q_p\hookrightarrow\C$.

\section{The indefinite case} \label{indefinite-case} 

In this section $B$ is an \emph{indefinite} quaternion algebra over $\Q$, whose discriminant $D\geq1$ is then a square-free product of an \emph{even} number of primes (if $D=1$ then $B\simeq\M_2(\Q)$). 

\subsection{Hecke algebras} \label{App1}

Let $\mathcal G$ be a group. For any subgroup $G\subset\mathcal G$ and any subsemigroup $S\subset\mathcal G$ such that $(G,S)$ is a Hecke pair in the sense of \cite[\S 1.1]{AS} we denote $\mathcal H(G,S)$ the Hecke algebra (over $\Z$) of the pair $(G,S)$, whose elements are combinations with integer coefficients of double cosets $T(s):=GsG$ for $s\in S$. If $\mathcal G=B^\times$ let $g\mapsto g^*:={\rm norm}(g)g^{-1}$ denote the main involution of $B^\times$, where ${\rm norm}:B^\times\rightarrow\Q$ is the norm map. Similarly, for any $S\subset B^\times$ as above let $S^*$ denote the image of $S$ under $g\mapsto g^*$. If $M$ is a left $\Z[S^*]$-module then the group $H^1(G,M)$ has a natural right action of $R(G,S)$ defined as follows. For $s\in S$ write $GsG=\coprod Gs_i$, then define functions $t_i:G\rightarrow G$ by the equations $Gs_i\gamma=Gs_j$ (for some $j$) and $g_i\gamma=t_i(\gamma)g_j$. The action on $H^1(G,M)$ is given at the level of cochains $c\in Z^1(G,M)$ by the formula
\[ \bigl(c|T(s)\bigr)(\gamma)=\sum_is_i^*c\bigl(t_i(\gamma)\bigr). \]
Fix a maximal order $R^{\rm max}$ in $B$. For every prime number $\ell\nmid D$ fix also an isomorphism of $\Q_\ell$-algebras
\[ i_\ell:B\otimes_\Q\Q_\ell\simeq\M_2(\Q_\ell) \]
in such a way that $i_\ell(R^{\rm max}\otimes_\Z\Z_\ell)=\M_2(\Z_\ell)$. For $x\in B$, we will occasionally write $i_\ell(x)$ in place of $i_\ell(x\otimes1)$. Fix also an integer $M\geq1$ prime to $D$ and a prime $p$ such that $p\nmid MD$. For any integer $r\geq0$ write $R_0^D(Mp^r)$ for the Eichler order of level $Mp^r$ contained in $R^{\rm max}$ and defined by the condition that $i_\ell(R_0^D(Mp^r)\otimes\Z_\ell)=R_\ell^{\rm loc}({\rm ord}_\ell(Mp^r))$ for all primes $\ell|Mp^r$ where, for every integer $n\geq 0$ and every prime $\ell$, we denote $R_\ell^{\rm loc}(n)$ the order of $\M_2(\Z_\ell)$ consisting of the matrices $\smallmat abcd$ with $c\equiv 0\pmod{\ell^n}$. Moreover, let $\Gamma_0^D(Mp^r)$ be the group of norm $1$ elements of $R_0^D(Mp^r)$ and let $\Gamma_r$ be the subgroup of $\Gamma_0^D(Mp^r)$ consisting of those $\gamma$ such that $i_\ell(\gamma)=\smallmat abcd$ with $a\equiv 1\pmod{Mp^r}$.  

For a prime $\ell\nmid D$ let $\Sigma_\ell^{\rm loc}$ denote the semigroup of elements in $R^{\rm max}\otimes\Z_\ell$ with non-zero norm, and for a prime $\ell|Mp$ and an integer $n\geq0$ let $\Sigma_\ell^{\rm loc}(\ell^n)\subset\Sigma_\ell^{\rm loc}$ be the inverse image under $i_\ell$ of the semigroup of matrices $\smallmat abcd\in\GL_2(\Q_\ell)\cap\M_2(\Z_\ell)$ with $a\equiv1\pmod{\ell^n}$ and $c\equiv0\pmod{\ell^n}$ (so $\Sigma_\ell^{\rm loc}(0)=\Sigma_\ell^{\rm loc}$). Then for every integer $r\geq0$ define the semigroups  
\[ \Sigma_r:=B^\times\cap\Bigg(\prod_{\ell|Mp}\Sigma_\ell^{\rm loc}({\rm ord}_\ell(Mp^r))\times\prod_{\ell\nmid Mp}\Sigma_\ell^{\rm loc}\Bigg) \]
and
\[ \Delta_r:=B^\times\cap\bigg(\Sigma_p^{\rm loc}(r)\times\prod_{\ell\neq p}\Sigma_\ell^{\rm loc}\bigg). \]
Finally, set $\Sigma_r^+:=\Sigma_r\cap B^+$ and $\Delta_r^+:=\Delta_r\cap B^+$ where $B^+$ is the subgroup of elements in $B^\times$ of positive norm. 

For every integer $n\geq1$ there is a Hecke operator $T_n=\sum_iT(\alpha_i)$ in $\mathcal H(\Gamma_r,\Sigma_r^+)$ and $\mathcal H(\Gamma_r,\Delta_r^+)$, where the sum is taken over all double cosets of the form $\Gamma_r\alpha_i\Gamma_r$ with $\alpha_i\in\Sigma_r^+$ and ${\rm norm}(\alpha_i)=n$. If $r\geq1$ we denote the Hecke operator $T_p$ by $U_p$. We also have operators $T_{n,n}\in\mathcal H(\Gamma_r,\Sigma_r^+)$ and $\mathcal H(\Gamma_r,\Delta_r^+)$ for integers $n\geq1$ prime to $MDp^r$, defined as follows. For every $n\in\Z$ with $(n,MDp^r)=1$ choose $\gamma_n\in\Gamma_0^D(Mp^r)$ such that $i_\ell(\gamma_n)\equiv\smallmat**0n\pmod{Mp^r}$ for all primes $\ell|Mp^r$ (use the Approximation Theorem: see, e.g., \cite[Theorem 5.2.10]{Mi}), then set $\delta_n:=n\gamma_n\in\Sigma_r^+\subset\Delta_r^+$ and define $T_{n,n}:=T_{\delta_n}$ in $\mathcal H(\Gamma_r,\Sigma_r^+)$ and $\mathcal H(\Gamma_r,\Delta_r^+)$, which is independent of the choice of $\gamma_n$. We also denote $\iota$ the Hecke operator $T_\beta$ in $\mathcal H(\Gamma_r,\Sigma_r)$ or $\mathcal H(\Gamma_r,\Delta_r)$ where $\beta$ is any element of $R_0^D(Mp^r)$ of norm $-1$ such that $i_\ell(\beta)\equiv\smallmat 100{-1}\pmod{Mp^r}$ for all primes $\ell|Mp^r$ (use again the Approximation Theorem). It can easily be checked that $\iota$ commutes with the elements $T_n$ and $T_{n,n}$ in $\mathcal H(\Gamma_r,\Sigma_r)$ and $\mathcal H(\Gamma_r,\Delta_r)$. Finally, recall that the $\mathcal H(\Gamma_r,\Sigma_r^+)$ are commutative rings generated (over $\Z$) by the Hecke operators $T_n$ and $T_{n,n}$ defined above. 

Let $\mathcal O$ be the ring of integers of a finite extension $F$ of $\Q_p$. For any integer $r\geq1$ denote $\mathfrak h_r^D$ the image of $\mathcal H(\Gamma_r,\Sigma_r^+)\otimes_\Z\mathcal O$ acting on the $\C$-vector space $S_2^\text{$D$-new}(\Gamma_1(MDp^r))$ of weight $2$ cusp forms on $\Gamma_1(MDp^r)$ which are new at all primes dividing $D$. 

\begin{remark}
Of course, the algebra $\mathfrak h_r^D$ depends on $\mathcal O$; however, since the field $F$ will always be clear in our applications, for simplicity we drop this dependence from the notation.
\end{remark}

Let $\mathfrak h_r^{D,\ord}$ denote the product of the localizations of $\mathfrak h_r^D$ where $U_p$ is invertible and write $e_r$ for the corresponding projector. Set $\mathfrak h_\infty^D:=\varprojlim_r\mathfrak h_r^D$ and let $e_\infty:=\varprojlim_re_r$ be the ordinary projector in $\mathfrak h_\infty^D$. Then define $\mathfrak h_\infty^{D,\ord}:=e_\infty\mathfrak h_\infty^D$, so that 
\[ \mathfrak h_\infty^{D,\ord}=\varprojlim_r\mathfrak h_r^{D,\ord}. \] 
Consider the Iwasawa algebra $\tilde\Lambda:=\mathcal O[\![\Z_p^\times]\!]$ with coefficients in $\mathcal O$ and denote $\gamma\mapsto[\gamma]$ the natural inclusion $\Z_p^\times\hookrightarrow\tilde\Lambda$ of group-like elements. Then $\mathfrak h_\infty^{D,\ord}$ is a finitely generated $\tilde\Lambda$-algebra. 

Following \cite{GS}, we are interested in defining a commutative $\tilde\Lambda$-algebra ${\mathfrak h^D_{\rm univ}}$ equipped with a canonical morphism of $\tilde{\Lambda}$-algebras $h:{\mathfrak h^D_{\rm univ}}\rightarrow{\mathfrak h_\infty^D}$. For this, we first consider the Hecke algebras $\mathcal H(\Gamma_r,\Delta_r)$ and $\mathcal H(\Gamma_r,\Delta_r^+)$ for integers $r\geq0$. Let $\Z'$ denote the subset of $\Z$ consisting of integers which are prime to $p$. For every $a\in\Z'$ 
choose $\gamma'_a\in\Gamma_0^D(Mp^r)$ such that $i_p(\gamma'_a)\equiv\smallmat**{**}a\pmod{p^r}$. Then $\delta_a':=a\gamma'_a\in\Delta_r^+$ and we can define the Hecke operator $[a]:=T_{\delta'_a}$ in $\mathcal H(\Gamma_r,\Delta_r^+)$ and $\mathcal H(\Gamma_r,\Delta_r)$, which does not depend on the choice of $\gamma'_a$. The maps $\Z'\rightarrow\mathcal H(\Gamma_r,\Delta_r)$ and $\Z'\rightarrow\mathcal H(\Gamma_r,\Delta_r^+)$ defined by $a\mapsto [a]$ are multiplicative, hence extend to ring homomorphisms $\Z[\Z']\rightarrow\mathcal H(\Gamma_r,\Delta_r)$ and $\Z[\Z']\rightarrow \mathcal H(\Gamma_r,\Delta_r^+)$. Since $\Z[\Z']$ also embeds naturally in $\tilde\Lambda$, we can form the $\tilde\Lambda$-algebras 
\[ \mathcal H(p^r):=\mathcal H(\Gamma_r,\Delta_r)\otimes_{\Z[\Z']}\tilde\Lambda,\qquad\mathcal H^+(p^r):=\mathcal H(\Gamma_r,\Delta_r^+)\otimes_{\Z[\Z']}\tilde\Lambda. \]
If $M$ is a $\Z_p[\Delta_r]$-module (respectively, a $\Z_p[\Delta_r^+]$-module) such that the action of $\Z'$ extends to a continuous action of $\Z_p^\times$ then for $i=0,1$ the action of the Hecke algebra $\mathcal H(\Gamma_r,\Delta_r)$ (respectively, $\mathcal H(\Gamma_r,\Delta_r^+)$) on $H^i(\Gamma_r,M)$ extends uniquely to a continuous action of $\mathcal H(p^r)$ (respectively, $\mathcal H^+(p^r)$). 

Now the Hecke pairs $(\Gamma_r,\Delta_r)$ and $(\Gamma_r,\Delta_r^+)$ are weakly compatible (according to \cite[Definition 2.1]{AS}) to $(\Gamma_0,\Delta_0)$ and $(\Gamma_0,\Delta_0^+)$, respectively. Hence, as explained in \cite[\S 2]{AS}, there are canonical surjective $\tilde\Lambda$-algebra homomorphisms 
\[ \rho_r:\mathcal H(1)\;\longepi\;\mathcal H(p^r),\qquad\rho_r^+:\mathcal H^+(1)\;\longepi\;\mathcal H^+(p^r) \] 
for all integers $r\geq1$. We let $\mathcal H(1)$ (respectively, $\mathcal H^+(1)$) act on $H^1(\Gamma_r,M)$ by composing the action of $\mathcal H(p^r)$ (respectively, $\mathcal H^+(p^r)$) with $\rho_r$ (respectively, $\rho_r^+$). 

Define the universal Hecke algebra ${\mathfrak h^D_{\rm univ}}$ as
\[ {\mathfrak h^D_{\rm univ}}:=\tilde\Lambda\bigl[\,\text{$T_n$ for every $n\geq1$ and $T_{n,n}$ for every $n\geq1$ with $(n,MD)=1$}\,\bigr]\subset\mathcal H^+(1). \] 
One can check that $\iota\in\mathcal H(\Gamma_0,\Delta_0)$ commutes with all the elements in ${\mathfrak h^D_{\rm univ}}$ and thus we can also consider the commutative Hecke algebra ${\mathfrak h^D_{\rm univ}}[\iota]\subset\mathcal H(1)$. 

The $\tilde\Lambda$-algebra ${\mathfrak h^D_{\rm univ}}$ acts compatibly on $S_2^D(\Gamma_1(Mp^r))$, in the sense that the diagram of $\tilde\Lambda$-algebras 
\[ \xymatrix{\mathfrak h^D_{\rm univ}\ar[r]^{\rho^+_r}\ar[rd]_{\rho^+_{r-1}}&\mathcal H^+(p^r)\ar@{->>}[d]\\
             & \mathcal H^+(p^{r-1})} \] 
commutes for all $r\geq1$ (here the vertical arrow is the map which arises from the weakly compatibility of the Hecke pairs $(\Gamma_r,\Delta_r^+)$ and $(\Gamma_{r-1},\Delta_{r-1}^+)$). The image of $\mathfrak h^D_{\rm univ}$ in the endomorphism algebra of $S_2^D(\Gamma_1(Mp^r))$ is canonically isomorphic to $\mathfrak h_r^D$ (note that if $n\in(\Z/Mp^r\Z)^\times$ then $n\in\Z'$ and $T_{n,n}$ is the image of $[n]\in{\mathfrak h^D_{\rm univ}}$), hence, by the universal property of the inverse limit, there exists a canonical morphism of $\tilde\Lambda$-algebras 
\[ h:\mathfrak h^D_{\rm univ}\longrightarrow\mathfrak h_\infty^D. \] 

\subsection{Hida families} \label{sec-arith-points}
 
Fix a non-zero normalized cusp form $f\in S_k\bigl(\Gamma_0(MDp^r),\epsilon\bigr)$ with $q$-expansion 
\[ f(q)=\sum_{n=1}^\infty a_nq^n, \] 
and suppose that $f$ is an eigenform for the action of the Hecke operators $T_n$ and $T_{n,n}$. Write $F_f$ for the field $\Q_p(a_n\mid n\geq1)$ generated over $\Q_p$ by the Fourier coefficients of $f$, let $\mathcal O_f$ denote the ring of integers of $F_f$ and let $\wp$ denote the maximal ideal of $\mathcal O_f$. In addition, assume that $f$ is an ordinary $p$-stabilized newform whose $\wp$-adic representation is residually irreducible and $p$-distinguished (see, e.g., \cite[\S 2]{Gh}). 

With notation as in \S \ref{App1}, take $\mathcal O=\mathcal O_f$, so that $\tilde\Lambda:=\mathcal O_f[\![\Z_p^\times]\!]$. Identify the group of $(p-1)$-st roots of unity in $\bar\Q_p$ with $T:=(\Z/p\Z)^\times$ via the Teichm\"uller character $\omega$ and set $W:=1+p\Z_p$, whence $\Z_p^\times\simeq T\times W$. This decomposition induces a decomposition $\tilde\Lambda=\mathcal O_f[T]\oplus\Lambda$ where $\Lambda$ is (non-canonically) isomorphic to the algebra of power series in one variable with coefficients in $\mathcal O_f$. A $\tilde\Lambda$-module $M$ inherits a canonical $\Lambda$-module structure via the inclusion $\Lambda\hookrightarrow\tilde\Lambda$. Finally, let $\mathcal L$ denote the fraction field of $\Lambda$.

There is a decomposition 
\[\mathfrak h_\infty^{1,\ord}\otimes_\Lambda\mathcal L\simeq\bigg(\bigoplus_{i\in I}\mathcal K_i\bigg)\oplus\mathcal N \]
where the $\mathcal K_i$ are finite field extensions of $\mathcal L$ (called the \emph{primitive components} of ${\mathfrak h_\infty^{1,\ord}}\otimes\mathcal L$), $I$ is a finite set and $\mathcal N$ is non-reduced. Denote $\mathcal K$ the primitive component through which the morphism associated with $f$ factors and let $\mathcal R$ be the integral closure of $\Lambda$ in $\mathcal K$. We call the induced map 
\[ f_\infty:{\mathfrak h_\infty^{1,\ord}}\longrightarrow \mathcal R \] 
the \emph{primitive morphism} associated with $f$. 

Now recall that, thanks to the Jacquet--Langlands correspondence, $S_k(\Gamma_r)$ is isomorphic to the subspace $S_k^\text{$D$-new}(\Gamma_1(MDp^r))$ of $S_k(\Gamma_1(MDp^r))$ consisting of those forms which are new at all the primes dividing $D$. Hence for all $r\geq1$ there is a canonical projection $\mathfrak h^1_r\rightarrow\mathfrak h_r^D$ which restricts to the ordinary parts for, and thus we get a canonical map ${\mathfrak h_\infty^{1,\ord}}\rightarrow{\mathfrak h_\infty^{D,\ord}}$. Now, as above, there is a splitting 
\[ {\mathfrak h_\infty^{D,\ord}}\otimes_\Lambda\mathcal L\simeq\bigg(\bigoplus_{j\in J}\mathcal F_j\bigg)\oplus\mathcal M \]
where $\mathcal F_j$ are finite field extensions of $\mathcal L$ and $\mathcal M$ is non-reduced. Since the morphism associated with $f$ factors through ${\mathfrak h_\infty^{D,\ord}}$, it must factor through some $\mathcal F\in\{\mathcal F_j\}_{j\in J}$ which is canonically isomorphic to $\mathcal K$. Summing up, we get a commutative diagram
\[ \xymatrix{{\mathfrak h_\infty^{1,\ord}}\ar[rr]^-{f_\infty}\ar[dr]&&\mathcal R\\
             &{\mathfrak h_\infty^{D,\ord}}\ar[ur]_-{f_\infty}} \]  
where we write $f_\infty$ also for the factoring map and the unlabeled arrow is the canonical projection considered before.  

For any topological $\mathcal O_f$-algebra $R$ let 
\[ \mathcal X (R):=\Hom_{\text{$\mathcal O_f$-alg}}^{\rm cont}(R,\bar\Q_p) \] 
denote the $\mathcal O_f$-module of continuous homomorphisms $R\rightarrow\bar\Q_p$ of $\mathcal O_f$-algebras.  
We call \emph{arithmetic homomorphisms} those $\kappa\in\mathcal X (\mathcal R)$ whose restriction to 
the canonical image of $W=1+p\Z_p$ in $\Lambda$ has the form $x\mapsto\epsilon(x)x^{k}$ for an integer $k\geq2$ and a finite order character $\epsilon$ of $W$. 
Write $\mathcal A(\mathcal R)$ for the subset of $\mathcal X(R)$ consisting of the arithmetic homomorphisms. The kernel $P_\kappa\in\Spec(\mathcal R)$ of a $\kappa\in\mathcal A(\mathcal R)$ is called an \emph{arithmetic prime}, and the residue field $F_\kappa:=\mathcal R_{P_\kappa}/P_{\kappa}\mathcal R_{P_\kappa}$ is a finite extension of $F_f$. The composition $W\rightarrow \mathcal R^\times\rightarrow F_\kappa^\times$ has the form $\gamma\mapsto \psi_\kappa(\gamma)\gamma^{k_\kappa}$ for a finite order character $\psi_\kappa:W\rightarrow F_\kappa^\times$ and an integer $k_\kappa\geq 2$. We call $\psi_\kappa$ the \emph{wild character} of $\kappa$ and $k_\kappa$ the \emph{weight} of $\kappa$.

Let $\kappa\in\cl A(\mathcal R)$. If $\kappa$ has weight $k=k_\kappa$ and character $\epsilon_\kappa$ then the composition 
\[ f_\kappa:=\kappa\circ f_\infty:{\mathfrak h_\infty^{D,\ord}}\longrightarrow\bar\Q_p \]
corresponds by duality to a modular form (denoted by the same symbol) 
\[ f_\kappa\in S_k\bigl(\Gamma_0(Np^{m_\kappa}),\epsilon_\kappa,F_\kappa\bigr) \]
of weight $k$, conductor divisible by $N$, level $\Gamma_0(Np^{m_\kappa})$ where $m_\kappa$ is the maximum between $1$ and the order at $p$ of the conductor of  $\psi_\kappa$ and character 
\[ \epsilon_\kappa:=\epsilon\psi_\kappa\omega^{-(k-2)}:\Z_p^\times\longrightarrow\bar\Q_p^\times. \] 
It is known that $f_\kappa\in S_k^\text{$D$-new}(\Gamma_1(MDp^r))$ for all $\kappa\in\mathcal A(\mathcal R)$.

\subsection{Modular forms on quaternion algebras} \label{sec2.4} 

For any commutative ring $R$ and any integer $n\geq0$ let 
\[ P_n(R):={\rm Sym}^n(R) \] 
denote the $R$-module of degree $n$ homogeneous polynomials in two variables with coefficients in $R$. It is equipped with a right action of the group $\GL_2(R)$ by the rule  
\[ (P|\gamma)(x,y):=P(ax+by,cx+dy)\qquad\text{for $\gamma=\smallmat abcd$}. \]
The $R$-linear dual $V_n(R)$ of $P_n(R)$ is then endowed with a left action of $\GL_2(R)$ by the formula
\[ (\gamma\phi)(P):=\phi(P|\gamma). \]
Finally, if $F$ is a splitting field for $B$ we may fix an isomorphism $i_F:B\otimes_\Q F\simeq\M_2(F)$, and then $P_n(F)$ (respectively, $V_n(F)$) is equipped with a right (respectively, left) action of $B^\times$ via $i_F$. In the applications, $F$ will be either $\Q_p$ or $\R$, so that we can (and do) choose $i_F$ to be $i_p$ or $i_\infty$, respectively. 

For every integer $r\geq0$ let $X_r$ denote the compact Shimura curve $\Gamma_r\backslash\mathcal H$ and write $\mathfrak h_{r,k}^D$ for the image of $\mathcal H(\Gamma_r,\Sigma_r^+)$ in the endomorphism algebra of $H^1\bigl(X_r,V_{k-2}(\C)\bigr)$. Let $g$ be a cusp form of level $\Gamma_r$ and weight $k$ which is a Hecke eigenform and denote $\lambda:\mathfrak h_{r,k}^D\rightarrow\C$ the corresponding ring homomorphism. Let $F_g$ be a subfield of $\C$ containing the image of $\lambda$, let $F/F_g$ be a field extension splitting $B$ and fix an isomorphism $i_F$ as above. Define 
\[ H^1\bigl(X_r,V_{k-2}(F)\bigr)^{g,\pm}:=\Big\{\xi\in H^1\big(X_r,V_{k-2}(F)\big)\;\big|\;\text{$\xi|T=\lambda(T)\xi$ for all $T\in\mathfrak h_{r,k}$ and $\xi|\iota=\pm\xi$}\Big\}. \] 
Thanks to a result of Matsushima and Shimura (\cite{MS}), we know that 
\[ \dim_L\Big(H^1\big(X_r,V_{k-2}(F)\big)^{g,\pm}\Big)=1. \]
Recall that there is a canonical isomorphism 
\[ H^1\bigl(X_r,V_{k-2}(\C)\bigr)\simeq H^1\bigl(\Gamma_r,V_{k-2}(\C)\bigr) \] 
which is equivariant for the action of the involution $\iota$. If $\tau\in\mathcal H$ then the complex vector space $H^1\bigl(\Gamma_r,V_{k-2}(\C)\bigr)^{g,\pm}$ is spanned by the projection on the $\pm$-eigenspace for $\iota$ of the cohomology class represented by the cocycle $\gamma\mapsto\omega(g)_\gamma$ given by  
\begin{equation} \label{eq0}
\omega(g)_\gamma\bigl(P(x,y)\bigr):=\int_\tau^{\gamma(\tau)}g(z)P(z,1)dz
\end{equation}
(the class does not depend on the choice of the base point $\tau\in\mathcal H$). For details, see \cite[\S 8.2]{Sh}. 

Now let $g\in S_k^{\text{$D$-new}}(\Gamma_1(MDp^r))$. The Jacquet--Langlands correspondence associates with $g$ a modular form $g^{\rm JL}$ of weight $k$ on $\Gamma_r$, which is well defined only up to a non-zero scalar factor. As above, let $F$ be a splitting 
field for $B$ containing the eigenvalues of the Hecke operators acting on $g$ (and so also on $g^{\rm JL}$). Then for any sign $\pm$ we may choose a multiple $g_\pm^{\rm JL}$ of $g^{\rm JL}$ in such a way the projection to the $\pm$-eigenspace of the cohomology class represented by the cocycle $\gamma\mapsto\omega(g_\pm^{\rm JL})_\gamma$ as in \eqref{eq0} generates $H^1\bigl(\Gamma_r,V_{k-2}(F)\bigr)^{g,\pm}$.   

Let $F$ be a subfield of $\C$ containing $F_f$ via the fixed embedding $\bar\Q_p\hookrightarrow\C$. By a slight abuse of notation, we use the symbol $\mathfrak h_{r,k}^D$ also to denote the image of $\mathcal H(\Gamma_r,\Sigma_r^+)\otimes_\Z\mathcal O_f$ in the endomorphism algebra of $H^1\bigl(X_r,V_{k-2}(F)\bigr)$.

For lack of a convenient reference, we prove a generalization of \cite[Theorem 7.2]{hida}. 

\begin{proposition} \label{7.2-hida} 
For every choice of sign $\pm$ the $\mathfrak h_{r,k}^D\otimes_{\mathcal O_f}F$-module $H^1(\Gamma_r,V_{k-2}(F))^\pm$ is free of rank $1$.
\end{proposition}

\begin{proof} The main result of \cite{MS} shows that the map 
\[ g\longmapsto\Big(\gamma\mapsto\Re\bigl(\omega(g)_\gamma\bigr)\Big), \] 
where $\Re$ denotes the real part of a complex number, induces an $\R$-linear isomorphism between $S_k(\Gamma_r)$ and $H^1(\Gamma_r,V_{k-2}(\R))$ (see also \cite[Theorem 8.4]{Sh}). One can rephrase this theorem by saying that there is an isomorphism 
\begin{equation} \label{MS}
H^1\bigl(\Gamma_r,V_{k-2}(\C)\bigr)\simeq S_k(\Gamma_r)\oplus\bar S_k(\Gamma_r)
\end{equation} 
where $\bar S_k(\Gamma_r)$ is the complex conjugate of the image of $S_k(\Gamma_r)$ in $H^1(\Gamma_r,V_{k-2}(\C))$ under the map $\omega$ introduced in \eqref{eq0} (cf. \cite[Section 2]{hida} and \cite[Theorem 6.2]{hida}). This isomorphism is compatible with the Hecke action. There is a hermitian positive definite bilinear pairing 
\[ (\,,):S_k(\Gamma_r)\times S_k(\Gamma_r)\longrightarrow \C \]
defined by 
\[ (f,g):=\int_{\Gamma\backslash \mathcal H}f(z)\overline{g(z)}y^{k-2}dxdy,\qquad z=x+iy \]  
(see, e.g., \cite[\S 8.2]{Sh}) which, using \eqref{MS}, induces isomorphisms 
\[ H^1\bigl(\Gamma_r,V_{k-2}(\C)\bigr)^\pm\simeq\Hom_\C\bigl(H^1(\Gamma_r,V_{k-2}(\C))^\mp,\C\bigr). \]
Recall that, thanks to \cite[Proposition 3.1]{hida-iwasawa}, there is a canonical isomorphism 
\[ \Hom_\C(\mathfrak h_{r,k}^1\otimes_{\mathcal O_f}\C,\C)\simeq S_k(\Gamma_1(MDp^r)). \] 
A morphism in the left hand side factors through $\mathfrak h_{r,k}^D$ if and only if the corresponding modular form is new at all the primes dividing $D$, and thus we obtain a non-canonical isomorphism 
\[ \Hom_\C(\mathfrak h_{r,k}^D\otimes_{\mathcal O_f}\C,\C)\simeq S_k(\Gamma_r) \] 
(here we fix an isomorphism $S_k^\text{$D$-new}(\Gamma_1(MDp^r))\simeq S_k(\Gamma_r)$). Therefore we get an isomorphism
\[ \mathfrak h_{r,k}^D\otimes_{\mathcal O_f}\C\simeq H^1\bigl(\Gamma_r,V_{k-2}(\C)\bigr)^\pm \] 
for each choice of sign $\pm$. Now the universal coefficient theorem shows that 
\[ H^1\bigl(\Gamma_r,V_{k-2}(\C)\bigr)^\pm\simeq H^1\bigl(\Gamma_r,V_{k-2}(F)\bigr)^\pm\otimes_F\C. \]
On the other hand, $\mathfrak h_{r,k}^D\otimes_{\mathcal O_f}\C\simeq(\mathfrak h_{r,k}^D\otimes_{\mathcal O_f}F)\otimes_F\C$, and the result follows because $\C$ is fully faithful over $F$. \end{proof}

\subsection{Measure-valued cohomology groups} \label{sec3.5} 

Fix a finite extension $F$ of $\Q_p$ and let $\mathcal O$ denote its ring of integers. 
Denote $\tilde\D(\mathcal O)$ the $\mathcal O$-module of $\mathcal O$-valued measures on 
$\Y:=\Z_p^2$ and by $\D(\mathcal O)$ the $\mathcal O$-submodule of $\tilde\D(\mathcal O)$ consisting of measures which are supported on the subset $\X$ of primitive vectors of $\Y$ (i.e., those vectors which are not divisible by $p$). Define 
\[ \W(\mathcal O):=H^1(\Gamma_0,\D(\mathcal O)). \] If $\mathcal O=\mathcal O_f$, we simply write $\tilde\D$, $\D$ and $\W$ for the 
corresponding objects. 
In what follows, we identify $\Gamma_r$ and $\Delta_r$ with their image in $\M_2(\Z_p)$ via $i_p$. Then, since the action of $\Z'$ extends to a continuous action of 
$\Z_p^\times$, it follows from the discussion in \S \ref{App1} that $\W(\mathcal O)$ is a $\tilde\Lambda$-module endowed with a canonical right action of ${\mathfrak h^D_{\rm univ}}[\iota]$. For any continuous $\tilde\Lambda$-algebra $R$ we also adopt the notation $\W(\mathcal O)_R:=\W(\mathcal O)\otimes_{\tilde\Lambda}R$. 

As an application of Shapiro's lemma, we get an isomorphism of $\Z_p$-modules 
\begin{equation} \label{shapiro}
\W(\mathcal O)\simeq\varprojlim_rH^1(\Gamma_r,\mathcal O)
\end{equation} 
(for details, see \cite[Proposition 7.6]{LRV1}). Applying \cite[Lemma 2.2 (b)]{AS}, we also see that \eqref{shapiro} is an isomorphism of ${\mathfrak h^D_{\rm univ}}[\iota]$-modules. 

As in \cite[Section 5]{as}, if $A$ is a compact $\Z_p$-module and $T:A\rightarrow A$ is a continuous homomorphism then the \emph{ordinary submodule of $A$ with respect to $T$} is 
\[ A^\ord:=\bigcap_{n=1}^\infty T^n(A). \]
It follows that $A^\ord$ is the largest submodule of $A$ on which $T$ acts invertibly. If $A$ is a profinite abelian group and $T$ is equal to a limit of operators on the finite quotients of $A$ then there is a canonical decomposition 
\[ A=A^\ord\oplus A^{\rm nil} \]
where the subgroup $A^{\rm nil}$ on which $T$ acts topologically nilpotently is the set of $a\in A$ such that $\lim_{n\rightarrow\infty}T^n(a)=0$ (see \cite[Proposition 2.3]{GS}). 

\begin{remark}
In the sequel, $T$ will always be the Hecke operator at $p$.
\end{remark}

Since each $H^1(\Gamma_r,\mathcal O)$ is a profinite group, so is $\W(\mathcal O)$ thanks to \eqref{shapiro}. Therefore, by specializing the above discussion to $A=\W$ and $T=T_p$, we can define 
\[ \W(\mathcal O)^\ord:=\bigcap_{n=1}^\infty\W(\mathcal O)|T_p^n \] 
and obtain a decomposition $\W(\mathcal O)=\W(\mathcal O)^\ord\oplus\W(\mathcal O)^{\rm nil}$.

The action of a Hecke operator $T=\Gamma_0\alpha\Gamma_0=\coprod_i\Gamma_0\alpha_i$ on a class $\boldsymbol\Phi\in\W(\mathcal O)$ 
can be described as follows. Fix a representative $\Phi$ of $\boldsymbol\Phi$. Then for any $\gamma\in\Gamma_0$ and any continuous function $\varphi$ on $\Y$ one has
\begin{equation} \label{Hecke-indef}
\begin{split}
(\Phi|T)_\gamma(\varphi)=\int_\X\varphi(x,y)d(\Phi|T)_\gamma&=\sum_i\int_\X\varphi(x,y)\alpha_i^*d\Phi_{t_i(\gamma)}=\sum_i\int_\X\varphi\bigl(\alpha_i^*(x,y)\bigr)d\Phi_{t_i(\gamma)}\\
&=\sum_i\int_{\alpha_i^*\X}\varphi(x,y)d\Phi_{t_i(\gamma)}=\sum_i\int_{\alpha_i^*\X\cap \X}\varphi(x,y)d\Phi_{t_i(\gamma)},
\end{split}
\end{equation} 
where the last equality follows from the fact that $\Phi$ is supported on $\X$ and, as usual, the $t_i(\gamma)$ are defined by the equations $\Gamma\alpha_i\gamma=\Gamma\alpha_j$ (for some $j$) and $\alpha_i\gamma=t_i(\gamma)\alpha_j$. 

We are interested in making the action of $T_p\in{\mathfrak h^D_{\rm univ}}$ explicit. By \cite[Theorem 5.3.5]{Mi}, the operator $T_p$ gives rise to a coset decomposition
\begin{equation} \label{Tp}
T_p=\coprod_{a\in\{0,\dots,p-1,\infty\}}\Gamma_0g_a
\end{equation}
where $i_p(g_\infty)=u_\infty\smallmat p001$ and $i_p(g_i)=u_i\smallmat 1{a_i}0{p}$ for some $u_\infty,u_i\in\GL_2(\Z_p)$ and some integers $a_i$ forming a complete system of representatives of $\Z/p\Z$. 

\subsection{Specialization maps} \label{specialization-subsec}

Before going on, let us observe a simple fact. Let $\epsilon:\Z_p^\times\rightarrow\bar\Q_p^\times$ be a character factoring through $(\Z/p^r\Z)^\times$ for an integer $r\geq1$, let $L$ be finite field extension of $\Q_p$ containing $F$ and the values of $\epsilon$ and let 
$\mathcal O_L$ be its ring of integers. 
Moreover, fix an element $\boldsymbol\nu\in H^1(\Gamma_0,\D(\mathcal O))$ and choose a representative $\nu$ of $\boldsymbol\nu$.

\begin{lemma} \label{lemma1}  
Fix an even integer $n\geq0$ and let $\U\subset\X$ be such that $\gamma\U=\U$ for all $\gamma\in\Gamma_r$. Then the function  
\[ \gamma\longmapsto\bigg(P\mapsto\int_\U\epsilon(y)P(x,y)d\nu_\gamma\bigg) \] 
defined on $\Gamma_r$ with values in $V_n(\mathcal O_L)$ is a $1$-cocycle whose class in 
$H^1(\Gamma_r,V_n(\mathcal O_L))$ does not depend on the choice of $\nu$.
\end{lemma}

\begin{proof} We know that 
\[ \int_\U\epsilon(y)P(x,y)d\nu_{\sigma\tau}=\int_\U\epsilon(y)P(x,y)d\nu_\sigma+\int_\U\epsilon(y)P(x,y)d\sigma\nu_\tau \] 
for all $\sigma,\tau\in\Gamma_r$. Since $\sigma^{-1}\U=\U$, the above equation also yields 
\[ \int_\U\epsilon(y)P(x,y)d\nu_{\sigma\tau}= \int_\U\epsilon(cx+dy)P(\sigma(x,y))d\nu_\sigma+\int_\U\epsilon(y)P(x,y)d\nu_\tau \]
where $\sigma=\smallmat abcd$. Since $\epsilon(cx+dy)=\epsilon(y)$, this  shows that the considered function is a $1$-cocycle. 

To complete the proof, let $\nu$ and $\nu'$ be representatives of $\boldsymbol\nu$; thus there exists $m\in\D(\mathcal O)$ 
such that $\nu_\gamma=\nu_\gamma'+\gamma(m)-m$ for all $\gamma\in\Gamma_0$. Then  
\begin{equation} \label{eq-indep}
\begin{split}
\int_\U\epsilon(y)P(x,y)d\nu_\gamma=&\int_\U\epsilon(y)P(x,y)d\nu'_\gamma\\
&+\int_{\gamma^{-1}\U}\epsilon(cx+dy)P(\gamma(x,y))dm-\int_\U\epsilon(y)P(x,y)dm 
\end{split}
\end{equation}
for all $\gamma=\smallmat abcd\in\Gamma_r$. Let $v\in V_n(\mathcal O_L)$ be defined by 
\[ v(P):=\int_\U\epsilon(y)P(x,y)dm. \] 
Then, by definition, $(\gamma\cdot v)(P(x,y))=v\bigl(P(\gamma(x,y))\bigr)$. Since $\gamma\U=\U$ for all $\gamma\in\Gamma_r$ and $\epsilon(cx+dy)=\epsilon(y)$, the result follows from \eqref{eq-indep}. \end{proof}

As in Lemma \ref{lemma1}, let $\epsilon$ be  a character of $\Z_p^\times$ factoring through $(\Z/p^r\Z)^\times$ for some integer $r\geq1$. Extend $\epsilon$ multiplicatively to $\Z_p$ by setting $\epsilon(p)=0$. For an integer $n\geq0$ and a field extension $L$ of $\Q_p$ containing the values of $\epsilon$ define the
\emph{specialization map}
\[ \rho_{n,\epsilon}=\rho_{n,\epsilon,L}:\D(\mathcal O)\longrightarrow V_n(\mathcal O_L) \]
by setting
\[ \rho_{n,\epsilon}(\nu)(P):=\int_{\Z_p\times\Z_p^\times}\epsilon(y)P(x,y)d\nu. \]
Since $\Z_p\times\Z_p^\times$ is stable under the action of $\Gamma_0^D(Mp)$, Lemma \ref{lemma1} ensures that if $\boldsymbol\Phi\in\W(\mathcal O)$ and $\Phi$ is a $1$-cocycle representing $\boldsymbol\Phi$ then the class in $H^1(\Gamma_r,V_n(\mathcal O_L))$ of the cocycle $\gamma\mapsto\rho_{n,\epsilon}(\Phi_\gamma)$ is independent of the choice of $\Phi$. This class will be denoted by $\rho_{n,\epsilon}(\boldsymbol\Phi)$. 

Let $\gamma=\smallmat abcd\in \GL_2(\Q_p)\cap\M_2(\Z_p)$ such that $a\in\Z_p^\times$ and $c\equiv0\pmod{p}$. Then for any $\nu\in\D(\mathcal O)$ one has
\[ \begin{split}
   \rho_{n,\epsilon}(\gamma^*\cdot\nu)(P)&=\int_\Y\chi_{\Z_p\times\Z_p^\times}(x,y)\epsilon(y)P(x,y)d(\gamma^*\cdot\nu)(x,y)\\
   &=\int_\Y\chi_{\Z_p\times\Z_p^\times}(dx-by,-cx+ay)\epsilon(-cx+ay)P(dx-by,-cx+ay)d\nu(x,y)\\
   &=\int_{\Z_p\times\Z_p^\times}\epsilon(a)\epsilon(y)P(dx-by,-cx+ay)d\nu(x,y)\\
   &=\epsilon(a)\rho_{n,\epsilon}(\nu)(P|\gamma^*)=\epsilon(a)\bigl(\gamma^*\cdot\rho_{n,\epsilon}(\nu)\bigr)(P),
   \end{split} \] 
whence
\[ \rho_{n,\epsilon}(\gamma^*\cdot\nu)=\epsilon(a)\bigl(\gamma^*\cdot\rho_{n,\epsilon}(\nu)\bigr). \]
Note that we have used the condition $p|c$ twice: first to obtain $\epsilon(-cx+ay)=\epsilon(ay)$ and then to get $\gamma^*(\Z_p\times\Z_p^\times)=\Z_p\times\Z_p^\times$.

Recall that if $p\nmid n$ then the Hecke operators $T_n$ and $T_{n,n}$ and the involution $\iota$ can be written as $\coprod_i\Gamma_0\alpha_i$ and $\coprod_i\Gamma_r\alpha_i$ for the same $\alpha_i$. Comparing with \eqref{Hecke-indef}, this shows that $\rho_{n,\epsilon}$ is compatible with the action of the Hecke operators $T_n$ and $T_{n,n}$ for $p\nmid n$ and with the action of $\iota$. For the operators $T_p$ and $U_p$, we observe that   
\[ \begin{pmatrix}p&0\\0&1\end{pmatrix}^{\!\!*}\X\cap(\Z_p\times\Z_p^\times)=\emptyset. \] 
Comparing again with \eqref{Hecke-indef}, we conclude that 
\begin{equation} \label{Tp-Up-indef}
\rho_{n,\epsilon}(\boldsymbol\Phi|T_p)=\rho_{n,\epsilon}(\boldsymbol\Phi)|U_p.
\end{equation}
Therefore, by passing to cohomology and restricting from $\Gamma_0$ to $\Gamma_r$, we obtain an $\mathfrak h^D_{\rm univ}[\iota]$-equivariant map
\begin{equation} \label{rho-n-eq}
\rho_{n,\epsilon}:\W(\mathcal O)\longrightarrow H^1\bigl(\Gamma_r,V_n(\mathcal O_L)\bigr), 
\end{equation}
denoted by the same symbol. 
Taking ordinary submodules in \eqref{rho-n-eq} yields an ${\mathfrak h^D_{\rm univ}}[\iota]$-equivariant map
\[ \rho_{n,\epsilon}^\ord:\W(\mathcal O)^\ord\longrightarrow H^1\bigl(\Gamma_r,V_n(\mathcal O_L)\bigr)^\ord. \]
Keeping in mind that $H^1(\Gamma_r,V_n(\mathcal O_L))$ is a compact $\Z_p$-module, we define
\[ H^1\bigl(\Gamma_r,V_n(L)\bigr)^\ord:=H^1\bigl(\Gamma_r,V_n(\mathcal O_L)\bigr)^\ord\otimes_{\mathcal O_L}L, \]
where the ordinary submodule on the right hand side is defined with respect to the operator $U_p$. It follows that $U_p$ acts invertibly on $H^1(\Gamma_r,V_n(L))^\ord$. 

\subsection{The Control Theorem} \label{sec-CT-ind}

\subsubsection{The kernel of the specialization map} \label{2.6.1} 

We begin by studying the specialization map introduced before and computing its kernel. The main result of this \S \ref{2.6.1} is Proposition \ref{propA4}. Recall that we fix the following data: an even integer $n\geq0$, a character $\epsilon:\Z_p^\times\rightarrow\bar\Q_p^\times$ factoring through $(\Z/p^r\Z)^\times$ for some integer $r\geq1$, and a finite extension $F$ of $\Q_p$ -- whose ring of integers we denote $\cO$ -- containing both $F_f$ and the values of $\epsilon$.

For any integer $m\geq1$ and any character $\chi:\Z_p^\times\rightarrow\bar\Q_p^\times$ define the function $\psi_{m,\chi}:\X\rightarrow\Z_p$ by 
\[ \psi_{m,\chi}((x,y)):=\begin{cases}\chi(y) & \text{if $(x,y)\in \U(m)$},\\[2mm]0 & \text{otherwise}.\end{cases} \]
If $\chi:\Z_p^\times\rightarrow\bar\Q_p^\times$ is a character then a continuous function $\varphi:\X\rightarrow\Z_p$ is \emph{homogeneous of degree $\chi$} if $\varphi(t(x,y))=\chi(t)\varphi(x,y)$ for all $t\in\Z_p^\times$. Clearly, $\psi_{m,\chi}$ is homogeneous of degree $\chi$ for all integers $m\geq 1$. 

\begin{lemma} \label{lemma-integration-II}
Let $\chi:\Z_p^\times\rightarrow\bar\Q_p^\times$ be the homomorphism defined by $\chi(t)=\epsilon(t)t^n$ for an integer $n\geq0$ and a character $\epsilon$ factoring through $(\Z/p^r\Z)^\times$ for some integer $r\geq1$. Let $\boldsymbol\Phi\in\W(\mathcal O)$ and let $\kappa\in\mathcal A(\tilde\Lambda)$ be of weight $k:=n+2$ and character $\epsilon$, so that the restriction of $\kappa$ to $\Z_p^\times$ coincides with $\chi$. Then the following conditions are equivalent:
\begin{enumerate}
\item $\boldsymbol\Phi\in P_\kappa\W(\mathcal O)$; 
\item $\boldsymbol\Phi$ can be represented by a cocycle $\Phi\in Z^1(\Gamma_0,P_\kappa\D(\mathcal O))$;
\item $\boldsymbol\Phi$ can be represented by a cocycle $\Phi\in Z^1(\Gamma_0,\D(\mathcal O))$ such that 
\[ \int_\X\varphi(x,y)d\Phi_\gamma(x,y)=0 \] 
for all homogeneous functions $\varphi:\X\rightarrow\Z_p$ of degree $\chi$ and all $\gamma\in\Gamma_0$;
\item $\boldsymbol\Phi$ can be represented by a cocycle $\Phi\in Z^1(\Gamma_0,\D(\mathcal O))$ such that 
\[ \int_\X\psi_{m,\chi}(x,y)d\Phi_\gamma(x,y)=0 \] 
for all integers $m\geq1$ and all $\gamma\in\Gamma_0$. 
\end{enumerate}
\end{lemma}

\begin{proof} The ideal $P_\kappa$ is principal, generated by $[\gamma]-\kappa(\gamma)$ where $\gamma$ is a topological generator of $1+p\Z_p$. By \cite[Lemma 1.2]{as}, it follows that $P_\kappa\W(\mathcal O)=H^1(\Gamma_0,P_\kappa\D(\mathcal O))$, and this shows the equivalence of (1) and (2). The equivalence of (2) and (3) follows directly from \cite[Lemma (6.3)]{AS}, and clearly (3) implies (4). To complete the proof, it remains to show that (4) implies (3), so suppose that (4) is true. For all $\gamma_1,\gamma_2\in\Gamma_0$ we have $\Phi_{\gamma_1\gamma_2}=\gamma_1\cdot\Phi_{\gamma_2}+\Phi_{\gamma_1}$, so, thanks to (4), there is an equality
\[ \int_\X\psi_{m,\chi}(x,y)d(\gamma_1\cdot\Phi_{\gamma_2})=\int_\X\psi_{m,\chi}(x,y)d\Phi_{\gamma_1\gamma_2}-\int_\X\psi_{m,\chi}(x,y)d\Phi_{\gamma_1}=0. \]
Therefore 
\begin{equation} \label{eq-lemma}
\int_\X\psi_{m,\chi}(\gamma_1(x,y))d\Phi_{\gamma_2}=0
\end{equation} 
for all $\gamma_1,\gamma_2\in\Gamma_0$. An easy argument shows that every function $\varphi$ which is homogeneous of degree $\chi$ is the uniform limit of a sequence of linear combinations of functions of the form $(x,y)\mapsto\psi_{m,\chi}(\gamma(x,y))$ for $m\geq1$ and $\gamma\in\Gamma_0$. Thus (3) follows from this fact and \eqref{eq-lemma}. \end{proof}

Now we need to slightly change notations in order to use arguments borrowed from \cite{as}. To do this, fix also a projective resolution $F_\bullet=\{F_k\}_k$ of $\Z$ by left $\mathcal O[\Gamma_0]$-modules. In the following it will be important to observe that, since $\cO[\Gamma_0]$ is a free $\cO[\Gamma_r]$-module (of finite rank), $F_\bullet$ is also a projective resolution of $\Z$ in the category of left $\cO[\Gamma_r]$-modules. Moreover, we notice that, under the assumptions of the above lemma, condition (4) in Lemma \ref{lemma-integration-II} is equivalent to the following 
\begin{itemize}
\item[(4')] $\boldsymbol\Phi$ can be represented by a cocycle $\Phi\in\Hom_{\Gamma_0}\bigl(F_1,\D(\mathcal O)\bigr)$ such that 
\[ \int_\X\psi_{m,\chi}(x,y)d\Phi(f)(x,y)=0 \] 
for all integers $m\geq1$ and all $f\in F_1$.
\end{itemize}
Now we are going to manipulate Lemma \ref{lemma-integration-II} by using the fixed resolution $F_\bullet$. For any integer $m\geq1$ define the open sets
\begin{equation} \label{U(m)}
\U(m):=\bigl\{(x,y)\in\X\mid x\equiv 0\pmod{p^m}\bigr\}.
\end{equation} 
and
\[ \mathbb V(m):=\bigl\{(x,y)\in\X\mid y\equiv0\pmod{p^m}\bigr\}. \] 
Let $V_n(\mathcal O)$ be the $\mathcal O$-linear dual of $P_n(\mathcal O)$. Define a map $\sigma_{n,\epsilon}^{(m)}:\D(\mathcal O)\rightarrow V_n(\mathcal O)$ as
\[ \sigma_{n,\epsilon}^{(m)}(\nu):=\bigg(P\mapsto\int_{\mathbb V(m)}\epsilon(x)P(x,y)d\nu\bigg). \]
Let $V_{n,\epsilon}^{(m)}(\mathcal O)$ denote the image of $\sigma_{n,\epsilon}^{(m)}$. One immediately verifies that $\gamma\mathbb V(m)=\mathbb V(m)$ for all $\gamma\in\Gamma_m$.The same argument as in Lemma \ref{lemma1} (simply replace the condition $\epsilon(cx+dy)=\epsilon(y)$ with the condition $\epsilon(ax+by)=\epsilon(x)$ for all $\gamma=\smallmat abcd\in\Gamma_r$, which is true for our choice of $\mathbb V(m)$) yields a well-defined map
\[ \sigma_{n,\epsilon}^{(m)}:\W(\mathcal O):=H^1\bigl(\Gamma_0,\D(\mathcal O)\bigr)\longrightarrow H^1\bigl(\Gamma_m,V_n(\mathcal O)\bigr). \]
This is obtained, as above, by fixing a representative $\Phi$ of $\boldsymbol{\Phi}\in\W(\mathcal O)$ and defining $\sigma_{n,\epsilon}^{(m)}(\boldsymbol\Phi)$ to be the class represented by the cocycle $\sigma_{n,\epsilon}^{(m)}(\Phi)$. 

The next results (Lemma \ref{lemma A1}--Lemma \ref{as-lemma}) explain how to translate in our setting the results of \cite[Section 7]{as}. 

\begin{lemma} \label{lemma A1}
Let $\boldsymbol{\Phi}\in\W(\mathcal O)$ and suppose that $\sigma_{n,\epsilon}^{(m)}\boldsymbol({\boldsymbol\Phi})=0$. Then, up to replacing $\mathcal O$ with a finite unramified extension, $\boldsymbol{\Phi}$ can be represented by a cocycle $\Phi$ such that $\sigma_{n,\epsilon}^{(m)}(\Phi)=0$ as a cochain in $\Hom_{\Gamma_m}(F_1,V_n(\mathcal O))$.
\end{lemma}

\begin{proof} Let $I$ denote the induced module 
\[ I:={\rm Ind}^{\Gamma_0}_{\Gamma_m}\bigl(V_{n,\epsilon}^{(m)}(\mathcal O)\bigr)=\bigl\{\phi:\Gamma_0\rightarrow V_{n,\epsilon}^{(m)}(\mathcal O)\mid\text{$\phi(\gamma x)=\gamma\phi(x)$ for all $x\in\Gamma$ and $\gamma\in\Gamma_m$}\bigr\}. \]  Here we view $\Gamma_m$ as a subgroup of $\GL_2(\Z_p)$. We make $I$ into a left $\Gamma_0$-module by the formula $y\phi(x):=\phi(yx)$ for all $x,y\in\Gamma_0$. We have a map
\begin{equation}\label{psi}
\psi:\D(\mathcal O)\longrightarrow I,\qquad\psi(\nu)(x):=\sigma_{n,\epsilon}^{({m)}}(x\nu).
\end{equation} 
Recall that, thanks to \cite[Lemma 7.1]{as}, if $\psi$ is surjective and $a\in B^1\bigl(\Gamma_0,V_{n,\epsilon}^{(m)}(\mathcal O)\bigr)$ -- where we adopt the usual notation $B^1(G,M)$ for the group of $1$-coboundaries of the discrete left $G$-module $M$ -- then we can choose a coboundary $b\in B^1(\Gamma_0,\D(\mathcal O))$ such that $\sigma_{n,\epsilon}^{(m)}(b)=a$ as cochains. (Note that we are adopting a slightly different formalism with respect to \cite{as}, where \emph{right} modules are used. However, 
\cite[Lemma 7.1]{as} is still true for \emph{left} modules if one converts a right action $m\mapsto m|x$ into a left one $m\mapsto x\cdot m$ by the formula $x\cdot m:=m|x^{-1}$.) Then to prove the result it is enough to show that $\psi$ is surjective because, 
if this is true, then $\sigma_{n,\epsilon}^{(m)}(\Phi)=a$ is a coboundary and hence, applying the above result, we may choose 
$b\in B^1(\Gamma,\D(\mathcal O))$ such that $\sigma_{n,\epsilon}^{(m)}(b)=a$, whence $\sigma_{n,\epsilon}^{(m)}(\Phi-b)=0$ as a cochain. 

We are thus reduced to showing that $\psi$ in \eqref{psi} is surjective. To do this, fix an $\mathcal O$-basis $\mathcal B$ of $V_{n,\epsilon}^{(m)}(\mathcal O)$. Since the image of $\psi$ is a $\Gamma_0$-submodule of $V_{n,\epsilon}^{(m)}(\mathcal O)$, 
it is enough to show that, for any $b\in\mathcal B$, the function sending $1$ to $b$ and $\Gamma_0-\Gamma_m$ to $0$ belongs to the image of $\psi$. Now for any $\nu\in \D(\mathcal O)$ and any $\gamma=\smallmat abcd\in\Gamma_0$ we have 
\[ \psi(\mu)(\gamma)(P)=\sigma_{n,\epsilon}^{(m)}(\gamma\nu)(P)=\int_{\mathbb V(m)}\epsilon(x)P(x,y)d\gamma\nu=\int_\X\chi_{\mathbb V(m)}(\gamma(x,y))\epsilon(ax+by)P(\gamma(x,y))d\nu \] 
where $\chi_{\mathbb V(m)}$ is the characteristic function of $\mathbb V(m)$. If $\nu$ is the Dirac measure supported at a point $(x_0,y_0)\in\mathbb V(m)$ then  
\[ \psi(\nu)(\gamma)(P)=\chi_{\mathbb V(m)}((ax_0+by_0,cx_0+dy_0))\epsilon(ax_0+by_0)P(ax_0+by_0,cx_0+dy_0). \]
Comparing with the definition of $\mathbb V(m)$, one immediately checks that $\psi(\nu)(\gamma)(P)\neq 0$ only if $c\equiv0\pmod{p^m}$, so only if $\gamma\in\Gamma_m$. Furthermore, for $\gamma=1$ we have 
\[ \psi(\nu)(1)(P)=\chi_{\mathbb V(m)}((x_0,y_0))\epsilon(x_0)P(x_0,y_0)=\epsilon(x_0)P(x_0,y_0). \] 
Fix now $b=\sigma_{n,\epsilon}^{(m)}(\nu')$ and write $b(x^iy^j)=a_{i,j}$. One immediately verifies that $a_{i,j}\in p^{mj}\mathcal O$. Since $\epsilon(x_0)\in\mathcal O^\times$, we are reduced to showing that for any set $\{a'_{i,j}\}\subset\mathcal O$ with $i,j$ non-negative integers such that $i+j=n$ there are $(x_0,y_0),\dots,(x_t,y_t)\in\mathcal O^\times\times\mathcal O$ such that $a'_{i,j}=\sum_{k=0}^t\alpha_{k}x_k^iy_k^j$ for suitable $\alpha_k\in\mathcal O$. To do this we show that, for example, we can find $(x_0,y_0),\dots,(x_t,y_t)$ as above with $t={n\choose2}$ such that the determinant of the matrix $\bigl(x_k^{n-i}y_k^i\bigr)_{i,k=0,\dots,t}$ belongs to $\mathcal O^\times$. Replacing $\mathcal O$ with the ring of integers of 
a sufficiently large unramified extension of $\Q_p$ with residue field $\mathbb F$ we can find $(\bar x_0,\bar y_0),\dots,(\bar x_t,\bar y_t)\in\mathbb F^\times\times\mathbb F$ with $t={n\choose2}$ such that the determinant of the matrix $\bigl(\bar x_k^{n-i}\bar y_k^i\bigr)_{i,k=0,\dots,t}$ is in $\F^\times$, and lifting these pairs to $\mathcal O^\times\times\mathcal O$ concludes the proof. \end{proof}

\begin{remark}
The condition that $\mathcal O$ be large enough is used only in the last step of the proof. In fact, we need the residue field of $\mathcal O$ to be sufficiently large. We believe that the result is still true without replacing $\mathcal O$, but at present we cannot find a simple proof of this fact.
\end{remark}

We need a general description of Hecke operators in terms of cochains; the following discussion is taken from \cite[p. 116]{DI}. Let $M$ be a $\mathcal O$-module endowed with a left action of $\GL_2(\Q_p)$. Fix an integer $m\geq 0$, let $\alpha\in\GL_2(\Q_p)$ be such that the groups $\Gamma_m\cap\alpha\Gamma_m\alpha^{-1}$ and $\alpha^{-1}\Gamma_m\alpha$ are commensurable with $\Gamma_m$, then set $\Gamma_m^{(\alpha)}:=\Gamma_m\cap\alpha\Gamma_m\alpha^{-1}$ and $\Gamma_{m}^{(\alpha^{-1})}:=\Gamma_m\cap\alpha^{-1}\Gamma_m\alpha$. The Hecke operator $T(\alpha)$ on $H^1(\Gamma_m,M)$ is defined as the composition 
\[ H^1(\Gamma_m,M)\xrightarrow{{\rm res}}H^1(\Gamma_m^{(\alpha)},M)\xrightarrow{{\rm conj}_\alpha}H^1\bigl(\Gamma_m^{(\alpha^{-1})},M\bigr)\xrightarrow{{\rm cores}}H^1(\Gamma_m,M), \] 
where 
\begin{itemize}
\item $\rm res$ and $\rm cores$ are the usual restriction and corestriction maps;
\item ${\rm conj}_\alpha$ is the map taking a cocycle $\gamma\mapsto c_\gamma$ to the cocycle $\gamma\mapsto\alpha^*c_{\alpha\gamma\alpha^{-1}}$.
\end{itemize} 
An easy formal computation (which can be found, e.g., in \cite[Proposition 3.1]{Wiese} for congruence subgroups) shows that this action agrees with the one already defined in \S \ref{App1}. This allows us to describe Hecke actions in terms of our fixed projective resolution, obtaining \cite[Formulae 4.3]{as}: if $z\in\Hom_{\Gamma_m}(F_k,M)$ and $\alpha$ is as above then $z|T(\alpha)$ is represented by the cochain 
\[ f\longmapsto\sum_i \alpha_i^*z\bigl(\tau(\gamma_i f)\bigr), \]
where 
\begin{itemize} 
\item the $\alpha_i$ are elements in $\GL_2(\Q_p)$ giving rise to the coset decomposition 
\[ \Gamma_m \alpha\Gamma_m=\coprod_i\Gamma_m\alpha_i; \]
\item the $\gamma_i$ are coset representatives for $(\alpha^{-1}\Gamma_m\alpha\cap\Gamma_m)\backslash\Gamma_m$; 
\item $\tau$ is a homotopy equivalence between the two resolutions $F_\bullet$ and $F_\bullet'$ of $\alpha^{-1}\Gamma_m\alpha$, 
where $F_\bullet '$ has the same underlying groups as $F_\bullet$ but the group action is defined by $(\alpha^{-1}\gamma\alpha)f_k':=\gamma f_k$.  
\end{itemize}
We begin by recalling the following application of the Approximation Theorem. 

\begin{lemma} \label{lemma*} 
Suppose that $m\geq1$. There are $\pi_m\in R_0^D(M)$ of norm $p^m$ and $w\in\Gamma_0$ such tha
 \begin{enumerate}
\item $\pi_mw$ normalizes $\Gamma_m$;
\item $(\pi_mw)^2\in p^m\Gamma_m$.
\end{enumerate} 
\end{lemma}

\begin{proof} To simplify notations, put $R:=R_0^D(Mp^m)$ and $R_\ell:=R\otimes_\Z\Z_\ell$ for all primes $\ell$. By the Approximation Theorem (\cite[Theorem 5.2.10]{Mi}), one can find elements $\pi_m\in R$ of norm $p^m$ and $w\in\Gamma_0$
such that $i_\ell(\pi_m)\equiv\smallmat 100{p^m}$ modulo $Mp^mR_\ell$ and $i_\ell(w)\equiv \smallmat 01{-1}0$ modulo $Mp^mR_\ell$ for all primes $\ell|Mp$. Then one has $i_\ell(\pi_mw)\equiv \smallmat 01{-p^m}0$ modulo $Mp^mR_\ell$ for all $\ell|Mp$. Since $i_\ell(\pi_mw)$ belongs to the normalizer of $R_\ell$ for all primes $\ell$, this shows (1). Furthermore, $(\pi_mw)^2\equiv p^m$ modulo $Mp^mR_\ell$ for all primes $\ell|Mp$. The element $p^{-m}(\pi_mw)^2$ is congruent to $1$ modulo $Mp^mR_\ell$ for all primes $\ell|Mp^m$, has determinant $1$ and belongs to $R_\ell$ for all $\ell$. Therefore $p^{-m}(\pi_mw)^2\in\Gamma_m$, and (2) is proved. \end{proof}

Thanks to part (2) of Lemma \ref{lemma*}, write 
\[ (\pi_mw)^2=p^m\gamma_0 \] 
with $\gamma_{0}\in \Gamma_m$. One can define the operator $X_m$ acting on $\Hom_{\Gamma_m}(F_\bullet, M)$ by the formula
\[ (z|X_m)(f):=\sum_{i=0}^{p^m-1}( w^{-1}\gamma_{m,i})^*z(w^{-1}\gamma_{m,i}f) \] 
where the $\gamma_{m,i}$ are coset representatives for $(\pi^{-m}\Gamma_m\pi^{m}\cap \Gamma_m)\backslash\Gamma_m$. If we let 
\[ \Gamma_m \pi_m\Gamma_m=\coprod_{i=0}^{p^m-1}\Gamma_mg_{m,i} \] 
then $\pi_m\gamma_{m,i}=g_{m,i}$ (cf., e.g., \cite[Proposition 3.1]{Sh}).

Now define the operator $W_m:=\Gamma_m\pi_mw\Gamma_m$. Then 
\begin{equation} \label{eqX_m}
W_mT_{p^m}=\Gamma_m\pi_mw\pi_m\Gamma_m=\Gamma_mp^m\gamma_0w^{-1}\Gamma_m=p^m
\Gamma_mw^{-1}\Gamma_m=p^mX_m\end{equation}
where the last equality follows from part (1) of Lemma \ref{lemma*}.

\begin{lemma} \label{lemmaA2}
In $\Phi\in\Hom_{\Gamma_0}(F_1,\D(\mathcal O))$ then $p^m\rho_{n,\epsilon}(\Phi)=\sigma_{n,\epsilon}^{(m)}(\Phi)|X_m$, where the equality holds at the cochain level. Hence, in cohomology, one has $p^m\rho_{n,\epsilon}(\boldsymbol\Phi)=\sigma_{n,\epsilon}^{(m)}(\boldsymbol\Phi)W_mT_{p^m}$, for all $\boldsymbol\Phi\in\W(\mathcal O)$.
\end{lemma}

\begin{proof} The last statement is immediate from \eqref{eqX_m} and the first one, which we are going to prove. A direct computation shows that 
\[ \begin{split}
   \bigl(\sigma_{n,\epsilon}^{(m)}(\Phi)|X_m\bigr)\bigl(P(x,y)\bigr)&=\left(\sum_{i=0}^{p^m-1}( w^{-1}\gamma_{m,i})^*\sigma_{n,\epsilon}^{(m)}\bigl(\Phi(w^{-1}\gamma_{m,i}f)\bigr)\right)\bigl(P(x,y)\bigr)\\
   &=\sum_{i=0}^{p^m-1}\int_{\mathbb V(m)}\epsilon(x)P\bigl((w^{-1}\gamma_{m,i})^*(x,y)\bigr)d\Phi(w^{-1}\gamma_{m,i}f)\\
   &=\sum_{i=0}^{p^m-1}\int_{\mathbb V(m)}\epsilon(x)P\big((w^{-1}\gamma_{m,i})^*(x,y)\big)w^{-1}\gamma_{m,i}d\Phi(f),
   \end{split} \] 
where the last equality follows from the $\Gamma_0$-equivariance of $\Phi$ (note that $w^{-1}\gamma_{m,i}\in\Gamma_0$). Now $(w^{-1}\gamma_{m,i})(w^{-1}\gamma_{m,i})^*=1$. We thus obtain 
\[ \bigl(\sigma_{n,\epsilon}^{(m)}(\Phi)|X_m\bigr)\bigl(P(x,y)\bigr)=\sum_{i=0}^{p^m-1}\int_{\gamma_{m,i}^{-1}w\mathbb V(m)}\epsilon\bigl(w^{-1}\gamma_{m,i}(x)\bigr)P(x,y)d\Phi(f). \]
But the family $\bigl\{\gamma_{m,i}^{-1}w\mathbb V(m)\bigr\}_i$ is a partition of $\Z_p\times\Z_p^\times$, hence
\[ \bigl(\sigma_{n,\epsilon}^{(m)}(\Phi)|X_m\bigr)\bigl(P(x,y)\bigr)=\int_{\Z_p\times\Z_p^\times}\epsilon(y)P(x,y)d\Phi(f), \] 
as was to be shown. \end{proof} 

\begin{lemma}\label{lemmaA2bis}
If $\Phi\in \Hom_{\Gamma_0}(F_1,\mathbb D(\mathcal O))$ then 
\[ \sigma_{n,\epsilon}^{(m)}\bigl(\Phi|T({w^{-1}\pi_mw})\bigr)=w^{-1}\pi_m^*\rho_{n,\epsilon}(\Phi'), \] 
where $\Phi'(f)=w\Phi(\tau(f))$ and 
$\tau:F_\bullet\rightarrow F_\bullet$ is a homotopy equivalence satisfying 
\[ \tau(\gamma f)=(w^{-1}\pi_m w)\gamma (w^{-1}\pi_{m}w)^{-1}\tau(f) \]
for $\gamma\in (w^{-1}\pi_mw)^{-1}\Gamma_0 (w^{-1}\pi_{m}w)\cap\Gamma_0=\Gamma_m$. 
\end{lemma}

\begin{proof} First we check that $w^{-1}\pi_m^*\rho_{n,\epsilon}^{(m)}(\Phi')$ is $\Gamma_m$-equivariant. Thanks to the $\Gamma_0$-equivariance of $\Phi$ and the definition of $\tau$, we have 
\[ \begin{split}\Phi'(\gamma f)&=w\Phi(\tau(\gamma f))=w w^{-1}\pi_mw\gamma w^{-1}(\pi_m)^{-1} w\Phi(\tau(f))\\
   &=(\pi_mw)\gamma (\pi_mw)^{-1}w\Phi(\tau(f))=(\pi_mw)\gamma (\pi_mw)^{-1}\Phi'(f).
   \end{split}\]
Since $\rho_{n,\epsilon}$ is $\Gamma_m$-equivariant, $\pi_m\pi^*_m=p^m$, and $\pi_mw$ normalizes $\Gamma_m$ by (1) in Lemma \ref{lemma*}, it follows that 
\[ w^{-1}\pi_m^*\rho_{n,\epsilon}(\Phi')(\gamma f)=w^{-1}\pi_m^*\rho_{n,\epsilon}\big((\pi_mw)\gamma(\pi_mw)^{-1}\Phi'(f)\big)=\gamma\rho_{n,\epsilon}\bigl(w^{-1}\pi^*\Phi'(f)\bigr). \]
Recall that $\Gamma_0\pi_m\Gamma_0=\coprod_i\Gamma_0g_{m,i}$ with $g_{m,i}=\pi_m\gamma_{m,i}$. Thus $\Gamma_0w^{-1}\pi_mw\Gamma_0$ is the disjoint union of the $\Gamma_0w^{-1}g_{i,m}w$ and $w^{-1}g_{m,i}w=w^{-1}\pi_mww^{-1}\gamma_{m,i}$. By definition, 
\[ \bigl(\Phi|T(w^{-1}\pi_mw)\bigr)(f)=\sum_i (w^{-1}g_{m,i}w)^*\Phi\bigl(\tau(w^{-1}\gamma_{m,i}w f)\bigr), \]
hence
\[ \begin{split}
   \sigma_{n,\epsilon}^{(m)}&\bigl(\Phi|T(w^{-1}\pi_mw)\bigr)(f)=\sum_i\int_{\mathbb V_m}\epsilon(x)P(x,y)d\Big(\!\bigl(w^{-1}g_{m,i}w\bigr)^*\Phi\bigl(\tau(w^{-1}\gamma_{m,i}w f)\bigr)\!\Big)\\
  &=\sum_i\int_{\X\cap {(w^{-1}g_{m,i}^*w)}^{-1}\mathbb V_m}\epsilon\bigl(w^{-1}g_{m,i}^*w(x)\bigr)P\bigl(w^{-1}g_{m,i}^*w(x,y)\bigr)
d\Big(\!\Phi\bigl(\tau(w^{-1}\gamma_{m,i}w f)\bigr)\!\Big).
   \end{split} \] 
Now $w\mathbb Y_m=\mathbb U_m$ and $g_{m,i}^*\X\cap\X\neq\emptyset$ if and only if $g_{m,i}=\pi_m$, in which case $\pi_m\X\subset\U_m$ and the corresponding $\gamma_{m,i}$ is equal to 1 (a similar argument will also be used in the proof of Lemma \ref{prop-kernel}). Hence $\X\cap {(w^{-1}g_{m,i}^*w)}^{-1}\mathbb V_m=w^{-1}\X$. Finally, notice that $\epsilon\bigl(w^{-1}\pi_{m}^*w(x)\bigr)=\epsilon(x)$ for $(x,y)$ in the domain of integration. Therefore the above sum is equal to 
\[ \int_{w^{-1}\X}\epsilon(x)P\bigl(w^{-1}\pi_m^*w(x,y)\bigr)d\bigl(\Phi(\tau(f))\bigr)=\int_\X\epsilon(y)P\bigl(w^{-1}\pi_m^*(x,y)\bigr)
d\bigl(w\Phi(\tau(f))\bigr) \] 
(note that $\epsilon(y)\neq 0$ if and only if $y\in\Z_p^\times$, so the integral is actually computed over $\Z_p\times\Z_p^\times$), and the proof is complete. \end{proof}

\begin{lemma} \label{lemmaA3}
If $\rho_{n,\epsilon}(\boldsymbol{\Phi})=0$ then $\sigma_{n,\epsilon}^{(m)}(\boldsymbol{\Phi}|T_p^m)=0$. 
\end{lemma}

\begin{proof} The operator $\Gamma_0w^{-1}\pi_m w\Gamma_0$ is nothing other than $\Gamma_0\pi_m\Gamma_0$ because $w\in\Gamma_0$. Let us represent $\boldsymbol\Phi$ by a cocycle $\Phi\in\Hom_{\Gamma_0}(F_1,\D(\mathcal O))$. By assumption, we can find a cochain $b\in\Hom_{\Gamma_m}(F_0,\D(\mathcal O))$ such that 
\[ b(df)(P(x,y))=\int_{\Z_p\times\Z_p^\times}\epsilon(y)P(x,y)d\Phi(f)(x,y) \] 
for all $f\in F_1$. It follows from Lemma \ref{lemmaA2bis} and the $\Gamma_0$-equivariance of $\Phi$ that $\sigma_{n,\epsilon}^{(m)}(\Phi|T(\pi^m))$ is represented by the functional sending a polynomial $P$ to 
\[ \int_{\Z_p\times\Z_p^\times}\epsilon(y)P\bigl(w^{-1}\pi^*(x,y)\bigr)dw\Phi(\tau(f))=b(dw\tau(f))\bigl(P(w^{-1}\pi^*(x,y))\bigr). \] 
This shows that $\sigma_{n,\epsilon}^{(m)}(\Phi|T(\pi_m))$ is represented by the coboundary $db'$ where $b'$ sends a polynomial 
$P$ to 
\[ \int_{\Z_p\times\Z_p^\times}\epsilon(y)P\bigl(w^{-1}\pi^*(x,y)\bigr)b(w\tau(f)). \] 
To complete the proof, we need to check that $b'$ is $\Gamma_m$-equivariant. This follows as in the first paragraph of the proof of Lemma \ref{lemmaA2bis} by formally replacing $\Phi$ with $b$ and $\Phi'$ with $b'$. \end{proof} 

\begin{lemma} \label{lemmaA4} 
If $\boldsymbol\Phi$ is ordinary and $\rho_{n,\epsilon}(\boldsymbol{\Phi})=0$ then $\sigma_{n,\epsilon}^{(m)}(\boldsymbol{\Phi})=0$.  
\end{lemma} 

\begin{proof} Choose $\boldsymbol{\Psi}$ such that $\boldsymbol{\Phi}=\boldsymbol{\Psi}|T_p^m$. Since $\rho_{n,\epsilon}(\boldsymbol{\Phi})=0$, we also have that $\boldsymbol{\Psi}$ is ordinary and $\rho_{n,\epsilon}(\boldsymbol{\Psi})=0$ (this argument will be used again in the proof of Lemma \ref{prop-kernel}). Lemma \ref{lemmaA3} then implies that $\sigma_{n,\epsilon}^{(m)}(\boldsymbol{\Phi})=\sigma_{n,\epsilon}^{(m)}(\boldsymbol{\Psi}|T_p^m)=0$. Finally, it can be checked that $b'$ is $\Gamma_m$-equivariant, which completes the proof. \end{proof}

\begin{lemma} \label{as-lemma}
Let $\boldsymbol\Phi\in\W(\mathcal O)^{\ord}$ and suppose that $\rho_{n,\epsilon}(\boldsymbol\Phi)=0$. If the residue field of $\mathcal O$ is sufficiently large then $\boldsymbol\Phi$ can be represented by a cocycle $\Phi$ such that $\rho_{n,\epsilon}(\Phi)=0$ as a cochain in $\Hom_{\Gamma_r}(F_1,V_n(\mathcal O))$.
\end{lemma}

\begin{proof} Since $\rho_{n,\epsilon}(\boldsymbol\Phi)=0$, it follows from Lemma \ref{lemmaA4} that $\sigma_{n,\epsilon}^{(m)}(\boldsymbol{\Phi})=0$. By Lemma \ref{lemma A1} one can choose a representative $\Phi$ of $\boldsymbol\Phi$ such that $\sigma_{n,\epsilon}^{(m)}({\Phi})=0$ as a cochain. Now Lemma \ref{lemmaA2} shows that 
\[ \rho_{n,\epsilon}(\Phi)=\rho_{n,\epsilon}(\Phi)|wT_p^{-m}=0 \] 
in $\Hom_{\Gamma_r}(F_1,V_n(\mathcal O))$. \end{proof}

We are now going to combine Lemma \ref{lemma-integration-II} with Lemma \ref{as-lemma} to study the kernel of the specialization map. 

\begin{lemma} \label{prop-kernel} 
Fix $\kappa\in\mathcal A(\tilde\Lambda)$ of weight $k$ and a character $\epsilon$ factoring through $(\Z/p^r\Z)^\times$ for some integer $r\geq1$. Set $n:=k-2$. Suppose that the residue field of $\mathcal O$ is sufficiently large, so that Lemma \ref{as-lemma} can be applied. Then the map $\rho_{n,\epsilon}^\ord$ induces an injective, $\mathfrak h^D_{\rm univ}[\iota]$-equivariant homomorphism of $\tilde\Lambda/P_\kappa\tilde\Lambda$-modules
\[ \rho_{n,\epsilon}^\ord:\W(\mathcal O)^\ord/P_\kappa\W(\mathcal O)^\ord\;\longmono\;H^1\bigl(\Gamma_r,V_n(\mathcal O)\bigr)^\ord. \] 
\end{lemma}

\begin{proof} Since the integrand $\epsilon(y)P(x,y)$ appearing in $\rho_{n,\epsilon}^\ord$ is homogeneous of degree $\kappa$, the inclusion $\ker(\rho_{n,\epsilon}^\ord)\supset P_\kappa\W(\mathcal O)^\ord$ follows from the implication $(1)\Rightarrow(3)$ in Lemma \ref{lemma-integration-II}.

Now we show the opposite inclusion. Let $\boldsymbol\Phi\in\ker(\rho_{n,\epsilon}^\ord)$ and represent it by a cocycle in $\Phi\in\Hom_{\Gamma_0}(F_1,\D(\mathcal O))$. Fix an integer $m\geq1$, choose $\boldsymbol\Psi\in\W(\mathcal O)^\ord$ such that $\boldsymbol\Psi|T_p^m=\boldsymbol\Phi$ (this is possible because $T_p$ induces an isomorphism on $\W(\mathcal O)^\ord$) and represent $\boldsymbol\Psi$ by a cocycle $\Psi$. Write $T_p^m$ as
\[ T_p^m=\coprod_i\Gamma_0 g_{m,i} \]
where the $g_{m,i}$ are suitable products of $m$ elements, not necessarily distinct, chosen in the set ${\{g_a\}}_{a=0,\dots,p-1,\infty}$ defined in \eqref{Tp}. Therefore, for all $f\in F_1$ we have 
\[ \int_\X\psi_{m,\kappa}(x,y)d\Phi(f)=\sum_i\int_\X\psi_{m,\kappa}\bigl(g_{m,i}^*(x,y)\bigr)d\Psi(f_i), \] 
where the $f_i$ are suitable elements in $F_1$ which can be made explicit using the definition of the Hecke actions given in terms of elements in $Z^1(\Gamma_0,\D(\mathcal O))$ and the identification between $Z^1(\Gamma_0,\D(\mathcal O))$ and $\Hom_{\Gamma_0}(F_1,\D(\mathcal O))$; however, we will not need this description in the following.
 
Now $\psi_{m,\kappa}\bigl(g_{m,i}^*(x,y)\bigr)=0$ unless $g_{m,i}=\smallmat 100{p^m}$. More precisely, since $\Psi(f_i)$ is supported on $\X$ and $\psi_{m,\kappa}$ is supported on the set $\U(m)$ defined in \eqref{U(m)}, the above integral does not vanish only if $\U(m)\cap\X\neq\emptyset$. An easy calculation shows that $g_{m,i}^*\X\cap\X\neq\emptyset$ if and only if $g_{m,i}=\smallmat 100{p^m}$, and in this case one has $\smallmat 100{p^m}^{\!*}\X\subset\U(m)$. Hence the $i$-th summand in the above sum is equal to
\[ \int_\X\psi_{\kappa,m}(p^mx,y)d\Psi(f_i)=\int_\X\kappa(y)d\Psi(f_i)=\int_\X\epsilon(y)y^nd\Psi(f_i). \]
On the other hand, by \eqref{Tp-Up-indef} there are equalities 
\[ 0=\rho_{n,\epsilon}^\ord(\boldsymbol\Phi)=\rho_{n,\epsilon}^\ord(\boldsymbol\Psi|T_p^m)=\rho_{n,\epsilon}^\ord(\boldsymbol\Psi)|U_p^m. \] 
Since $U_p$ acts invertibly on $H^1(\Gamma_r,V_n(F_f))^\ord$, it follows that $\rho_{n,\epsilon}^\ord(\boldsymbol\Psi)=0$. 

By invoking Lemma \ref{as-lemma}, we choose a representative $\Psi$ of $\boldsymbol\Psi$ such that $\rho_{n,\epsilon}(\Psi)=0$ in $\Hom_{\Gamma_r}(F_1,V_n(F_\kappa))$, and then we conclude that 
\[ \int_\X\epsilon(y)y^nd\Psi(f)=0 \] 
for all $f\in F_1$. Define $\Phi_m:=\Psi|T_p^m$, which is a representative of $\boldsymbol\Phi=\boldsymbol\Psi|T_p^m$. Then, since the map $\rho_{n,\epsilon}$ is compatible with the action of $T_p$, we conclude that $\rho_{n,\epsilon}(\Phi_m)=0$ in $\Hom_{\Gamma_r}(F_1,\D(\mathcal O))$, and so we get 
\[ \int_\X \psi_{m,\kappa}(x,y)d\Phi_m(f)=0 \] 
for all $f\in F_1$. Since $\Hom_{\Gamma_0}(F_1,\D(\mathcal O))$ is compact, we can assume that the sequence ${(\Psi_m)}_{m\geq1}$ has a limit, which we denote $\Psi_\infty$. But the coboundary map is continuous, hence $\Psi_\infty$ is also a cocycle. Since coboundaries form a compact subspace of the group of cocycles, $\Psi_\infty$ still represents $\boldsymbol\Psi$. Finally, since the topology on the space of measures is induced by pointwise convergence on continuous functions, we see that 
\[ \int_\X \psi_{m,\kappa}(x,y)d\Phi_\infty(f)=0 \] 
for all $f\in F_1$ and all integers $m\geq1$. From the equivalence of conditions (4') above and (4) in Lemma \ref{lemma-integration-II} and the implication $(4)\Rightarrow(1)$ in Lemma \ref{lemma-integration-II} it follows that $\boldsymbol\Phi\in P_\kappa\W(\mathcal O)^\ord$, as was to be shown. \end{proof}

\begin{lemma}
Fix $\kappa\in\mathcal A(\tilde\Lambda)$ of weight $k$ and a character $\epsilon$ factoring through $(\Z/p^r\Z)^\times$ for some integer $r\geq1$. Set $n:=k-2$. The map $\rho_{n,\epsilon}^\ord$ induces an injective, $\mathfrak h^D_{\rm univ}[\iota]$-equivariant homomorphism of $\tilde\Lambda/P_\kappa\tilde\Lambda$-modules
\[ \rho_{n,\epsilon}^\ord:\W^\ord/P_\kappa\W^\ord\;\longmono\;H^1\bigl(\Gamma_r,V_n(\mathcal O)\bigr)^\ord. \] 
\end{lemma}

\begin{proof} Fix $\mathcal O$ sufficiently large so that the above lemma can be applied. The extension $\mathcal O_f\subset\mathcal O$ is fully faithful. Moreover, there are isomorphisms  
\[ \W\otimes_{\mathcal O_f}\mathcal O=H^1(\Gamma_0,\D)\otimes_{\mathcal O_f}\mathcal O\simeq H^1(\Gamma_0\backslash\mathcal H,\D)\otimes_{\mathcal O_f}\mathcal O\simeq H^1(\Gamma_0\backslash\mathcal H,\D(\mathcal O))\simeq\W(\mathcal O) \] 
and 
\[ H^1\bigl(\Gamma_r,V_n(\mathcal O_f)\bigr)\otimes_{\mathcal O_f}\mathcal O\simeq H^1\bigl(\Gamma_r,V_n(\mathcal O)\bigr) \] coming from the universal coefficient theorem for cohomology (see \cite[Theorem 15.3]{br}) which are compatible with the 
action of Hecke operators. The result follows. \end{proof}

For any $\mathcal O_f[\mathfrak h^D_{\rm univ}]$-module $M$ and any ring homomorphism $\vartheta:\mathfrak h^D_{\rm univ}\rightarrow R$ define
\[ M^\vartheta:=M\otimes_{\fr h_{\rm univ}^D}R, \] 
the tensor product being taken with respect to $\vartheta$. Let $\kappa\in\mathcal A(\mathcal R)$ and define the $F_\kappa$-vector space
\[ \W_\kappa^\ord:=\Big(H^1\bigl(\Gamma_{m_\kappa},V_{n_\kappa}(F_\kappa)\bigr)^\ord\Big)^{f_\kappa}. \] 
Let $h_\kappa$ denote the composition of $h:\mathfrak h^D_{\rm univ}\rightarrow\mathfrak h_\infty^{D,\ord}$ with the canonical map $\mathfrak h_\infty^{D,\ord}\rightarrow\mathcal R_{P_\kappa}$. Then we can consider the $\mathcal R_{P_\kappa}$-submodule $\W_{h_\kappa}^\ord$ of $\W^\ord_{\mathcal R_{P_\kappa}}$ defined by 
\[ \W_{h_\kappa}^\ord:=\bigl(\W^\ord_{\mathcal R_{P_\kappa}}\bigr)^{h_\kappa}. \]
The action of the involution $\iota$ on a $\Q_p$-vector space $M$ induces a splitting $M=M^+\oplus M^-$, where $M^\pm$ are the $\pm$-eigenspaces for $\iota$. 

\begin{proposition} \label{propA4}
Let $\kappa\in\mathcal A(\mathcal R)$ have weight $k=k_\kappa$ and character $\epsilon=\epsilon_\kappa$. Set $n:=k-2$. The map $\rho_{n,\epsilon}^\ord$ of Lemma \ref{prop-kernel} induces an injective homomorphism of $F_\kappa$-vector spaces
\[ \rho_\kappa:\W_{h_\kappa}^{\ord,\pm}/P_\kappa\W_{h_\kappa}^{\ord,\pm}\;\longmono\;\W_\kappa^{\ord,\pm}. \]
\end{proposition}

\begin{proof} Define $\p_\kappa:=P_\kappa\cap\Lambda$. Using the canonical projection $\tilde\Lambda\rightarrow\Lambda$ we may view $\p_\kappa$ as an element in $\mathcal A(\tilde\Lambda)$. We deduce from Lemma \ref{prop-kernel} the existence of an injective homomorphism of $\tilde\Lambda/\p_\kappa\tilde\Lambda$-modules 
\[ \rho_{n,\epsilon}^\ord:\W^\ord/{\p_\kappa}\W^\ord\;\longmono\;H^1\bigl(\Gamma_r,V_n(\mathcal O_f)\bigr)^\ord. \]
Now $\tilde\Lambda/\p_\kappa\tilde\Lambda\simeq\Lambda/\p_\kappa\Lambda$, thus we get an injective homomorphism of $\Lambda$-modules  
\[ \rho_{n,\epsilon}^\ord:\W^\ord_\Lambda/{\p_\kappa}\W^\ord_\Lambda\;\longmono\;H^1\bigl(\Gamma_r,V_n(\mathcal O_f)\bigr)^\ord. \]
Since $\Lambda_{\p_\kappa}$ is flat over $\Lambda$, we also obtain an injective homomorphism of $\Lambda_{\p_\kappa}/\p_\kappa\Lambda_{\p_\kappa}$-vector spaces  
\[ \rho_{n,\epsilon}^\ord:\W_{\Lambda_{\p_\kappa}}^\ord/{\p_\kappa}\W_{\Lambda_{\p_\kappa}}^\ord\;\longmono\;H^1\bigl(\Gamma_r,V_n(F_f)\bigr)^\ord. \]
Now $\mathcal R_{P_\kappa}$ and $\Lambda _{\p_\kappa}$ are normal domains and $\mathcal R_{P_\kappa}$, being unramified over $\Lambda_{\p_\kappa}$ thanks to \cite[Corollary 1.4]{hida-inv}, is flat over $\Lambda_{\p_\kappa}$. Using the universal coefficient theorem for cohomology (see \cite[Theorem 15.3]{br}) and recalling that $H^1(\Gamma_r,V_n(F))$ is canonically isomorphic to $H^1(\Gamma_r\backslash\mathcal H,V_n(F))$, we get an isomorphism 
\[ H^1\bigl(\Gamma_r,V_n(F_f)\bigr)\otimes_{F_f}F_\kappa\simeq H^1\bigl(\Gamma_r,V_n(F_\kappa)\bigr), \] 
from which we deduce an injective homomorphism of $F_\kappa$-vector spaces
\[ \rho_\kappa^\ord:\W ^\ord_{\mathcal R_{P_\kappa}}/P_\kappa\W ^\ord_{\mathcal R_{P_\kappa}}\;\longmono\;H^1\bigl(\Gamma_r,V_n(F_\kappa)\bigr)^\ord. \] 
Thanks to the $\mathfrak h^D_{\rm univ}[\iota]$-equivariance of $\rho_{n,\epsilon}^\ord$, restricting $\rho_\kappa^\ord$ to $\W_{h_\kappa}^{\ord,\pm}$ gives the searched-for injection. \end{proof}

\subsubsection{Dimension bounds} Now we compute the dimensions of the source and the target of the specialization map. Recall that $\mathcal L$ (respectively, $\mathcal K$) is the fraction field of $\Lambda$ (respectively, $\mathcal R$) and that there is a canonical decomposition 
\begin{equation} \label{decomp}
\mathfrak h_\infty^{D,\ord}\otimes_\Lambda\mathcal L\simeq\mathcal K\oplus\mathcal N
\end{equation}
where $\mathcal N$ is a sum of finitely many fields and of a non-reduced part. Define
\[ h_\mathcal R:\mathfrak h^D_{\rm univ}\overset{h}\longrightarrow\mathfrak h_\infty^D\longrightarrow\mathfrak h^{D,\ord}_\infty\xrightarrow{f_\infty}\mathcal R. \]

\begin{proposition} \label{prop-2}
The module $\W^\ord_\mathcal K$ is a $2$-dimensional vector space over $\mathcal K$ and each eigenmodule for $\iota$ has dimension $1$. Moreover, the action of $\mathfrak h^D_{\rm univ}$ on $\W^\ord_\mathcal K$ factors through $h_\mathcal R$. 
\end{proposition}

\begin{proof} For every integer $r\geq1$ let $X_r$ denote the compact Shimura curve $\Gamma_r\backslash\mathcal H$. Define the $p$-divisible abelian group 
\[ \mathcal V:=\varinjlim_r H^1(X_r,F_f/\mathcal O_f)^\ord=\varinjlim_rH^1(\Gamma_r,F_f/\mathcal O_f)^\ord, \]
where the inductive limit is taken with respect to the restriction maps. The Hecke operators $T_n$, $T_{n,n}$ and the involution $\iota$ act naturally on $\mathcal V$, since the Hecke action is compatible with the restriction maps (see \cite[(2.9 a,b) and (3.5)]{hida}). Consider the eigenmodules $\mathcal V^\pm$ for $\iota$ and define $V$ to be the Pontryagin dual of one of $\mathcal V^\pm$. Thanks to \cite[Corollary 10.4]{hida}, we know that $V$ is free of finite rank over $\Lambda$.

Replacing \cite[eq. (7.6)]{hida} with Proposition \ref{7.2-hida}, we can mimic the proof of \cite[Theorem 12.1]{hida} and show that there is an isomorphism $V\otimes_\Lambda\Lambda_{P_\kappa}\simeq\mathfrak h_\infty^{D,\ord}\otimes_\Lambda\Lambda_{P_\kappa}$ for all $\kappa\in\mathcal A(\Lambda)$. Now the Pontryagin dual of $\mathcal V$ is identified with the inverse limit of the cohomology groups $H^1(\Gamma_r,\mathcal O_f)^\ord$ with respect to the corestriction maps. On the other hand, Shapiro's Lemma \eqref{shapiro} is equivariant for the action of $\mathfrak h^D_{\rm univ}$. Therefore $\W^\ord\otimes_\Lambda\Lambda_{P_\kappa}$ is free of rank $2$ over $\mathfrak h_\infty^{D,\ord}\otimes_\Lambda\Lambda_{P_\kappa}$ for all $\kappa\in\mathcal A(\Lambda)$, and each eigenmodule for $\iota$ is free of rank $1$. Hence it follows that $\W_\mathcal L^\ord$ is free of rank $2$ over $\mathfrak h_\infty^{D,\ord}\otimes_\Lambda\mathcal L$, and each eigenspace for $\iota$ is free of rank $1$. In light of decomposition \eqref{decomp}, the proof is complete. \end{proof}

\subsubsection{Control Theorem} We are now ready to state and prove the main result of this section. 

\begin{theorem} \label{control-thm}
For every $\kappa\in\mathcal A(\mathcal R)$ the map $\rho_\kappa$ of Proposition \ref{propA4} induces an isomorphism of $1$-dimensional $F_\kappa$-vector spaces
\[ \rho_\kappa:\W^{\ord,\pm}_{h_\kappa}\big/P_\kappa\W_{h_\kappa}^{\ord,\pm}\overset\simeq\longrightarrow\W^{\ord,\pm}_\kappa. \]
\end{theorem}

\begin{proof} Thanks to Proposition \ref{propA4}, we only need to show that
\[ \dim_{F_\kappa}(\W_{h_\kappa}^\ord/P_\kappa\W_{h_\kappa}^\ord)\geq\dim_{F_\kappa}(\W_{\kappa}^\ord). \] 
Since $\dim_{F_\kappa}(\W_{\kappa}^\ord)=2$ as observed in \S \ref{sec2.4}, we are reduced to proving that 
\[ \dim_{F_\kappa}(\W_{h_\kappa}^\ord/P_\kappa\W_{h_\kappa}^\ord)\geq2. \] 
However, Proposition \ref{prop-2} shows that the intersection of $\W_{\mathcal R_\kappa}^\ord$ with the $h_\mathcal R$-eigenmodule of $\W^\ord_\mathcal K$ is a free $\mathcal R_{P_\kappa}$-module of rank $2$, and we are done. \end{proof}

\section{The definite case} \label{def-sec}

In this section $B$ is a \emph{definite} quaternion algebra over $\Q$, whose discriminant $D>1$ is then a square-free product of an \emph{odd} number of primes.

\subsection{Modular forms on quaternion algebras} \label{definite-forms-subsec}

We will often use notations and results from Section \ref{indefinite-case}. For all primes $\ell\nmid D$ fix isomorphisms of $\Q_\ell$-agebras 
\[ i_\ell:B\otimes_\Q\Q_\ell\overset\simeq\longrightarrow\M_2(\Q_\ell) \] 
and a maximal order $R^{\rm max}$ in $B$ such that $i_\ell(R^{\rm max}\otimes_\Z\Z_\ell)=\M_2(\Z_\ell)$. For all primes $\ell\nmid D$ and all integers $r\geq1$ choose Eichler orders $R_0^D(Mp^r)\subset B$ of level $Mp^r$ such that $i_\ell(R_0^D(Mp^r)\otimes\Z_\ell)$ is the order of $\M_2(\Z_\ell)$ consisting of the matrices $\smallmat abcd$ with $c\equiv 0\pmod{Mp^r}$. For every integer $r\geq0$ define the compact open subgroup $U_r$ of $\hat B^\times$ as 
\[ U_r:=\Big\{{(x_\ell)}_\ell\in\hat R_0^D(Mp^r)^\times\;\big|\;\text{$i_\ell(x_\ell)=\smallmat abcd$ with $a\equiv1\pmod{\ell^{{\rm ord}_\ell(Mp^r)}}$ for all $\ell|Mp^r$}\Big\}. \]
We begin by recalling the definition of modular forms on $B$ which can be found, e.g., in \cite[Definition 2.1]{BD}; references are \cite[Section 2]{BD}, \cite[Section 4]{Bu} and \cite[Section 2]{hida}. If $A$ is a $\Z_p$-module equipped with a left linear action of $\M_2(\Z_p)\cap\GL_2(\Q_p)$ and $U$ is a compact open subgroup of $\hat B^\times$ then an \emph{$A$-valued modular form on $B$ of level $U$} is an element of the $A$-module $S(U,A)$ of functions 
\[ s:\hat B^\times\longrightarrow A \]
such that 
\[ g(bgu)=i_p(u_p)^{-1}s(g)\qquad\text{for all $b\in B^\times$, $g\in\hat B^\times$ and $u\in U$}, \] 
where $u_p$ denotes the $p$-component of $u$. Therefore a modular form in $S(U,A)$ is completely determined by its values on the finite set 
\[ X(U):=B^\times\backslash\widehat B^\times/U. \]  
Finally, for $U=U_r$ set $X_r:=X(U_r)$ and $S_k(U_r,A):=S(U_r,V_{k-2}(A))$.  

\begin{remark}
As in \cite{BD}, the above definition works for $\Z_p$-modules endowed with a \emph{left} linear action of $\GL_2(\Q_p)\cap\M_2(\Z_p)$. The definition which can be found in \cite{Bu} uses, on the contrary, \emph{right} actions on $A$. Of course, the two definitions are compatible, as one sees by turning the right action in \cite{Bu} into a left one via the formula $\gamma\cdot a:=a|\gamma^*$.  
\end{remark} 

\begin{remark} 
The definition in \cite{hida} looks different from the ones in in \cite{Bu} and \cite{BD}. The point is that in \cite{hida} the weight action on polynomials is concentrated in the archimedean place, while the above definition makes use of the place at $p$. However, the two notions are equivalent whenever we fix an embedding $\bar\Q_p\hookrightarrow\C$. For details, see \cite[Section 4]{Bu} and the references quoted there. 
\end{remark}

\subsection{Hecke algebras} \label{App1-def} 

We review the theory of \S \ref{App1} in the adelic language, which is more suitable for applications to definite quaternion algebras (and for generalizations of the theory to the case of totally real fields).

For any integer $r\geq0$ define 
\[ \Sigma_r:=\Big\{{(x_\ell)}_\ell\in\hat R_0^D(Mp^r)\cap\hat B^\times\;\big|\;\text{$i_\ell(x_\ell)=\smallmat abcd$ with $a\equiv1\pmod{\ell^{{\rm ord}_\ell(Mp^r)}}$ for all $\ell|Mp^r$}\Big\} \]
and
\[ \Delta_r:=\bigl\{{(x_\ell)}_\ell\in\hat R^{\rm max}\cap\hat B^\times\mid\text{$i_\ell(x_p)=\smallmat abcd$ with $a\equiv1\pmod{p^r}$}\bigr\}. \]
Now fix an integer $r\geq0$. For every integer $n\geq2$ there is a Hecke operator $T_n=\sum_iT(\alpha_i)$ in $\mathcal H(U_r,\Sigma_r)$ and $\mathcal H(U_r,\Delta_r)$, where the sum is taken over all double cosets $U_r\alpha_iU_r$ with $\alpha_i\in\Sigma_r$ and ${\rm norm}(\alpha_i)\hat\Z=n\hat\Z$ (here ${\rm norm}:\hat B\rightarrow\hat\Q$ is the adelization of the norm map). There are also Hecke operators $T_{n,n}$ in $\mathcal H(U_r,\Sigma_r)$ and $\mathcal H(U_r,\Delta_r)$ for integers $n\geq1$ prime to $MDp^r$, defined as follows. For any such $n$ choose $\nu\in\hat\Z\cap\hat\Q^\times$ such that $n\hat\Z=\nu\hat\Z$ and $\nu-1\in Mp^r\hat\Z$, then define $T_{n,n}=U_r\nu U_r$. Finally, recall that $\mathcal H(U_r,\Sigma_r)$ is the commutative algebra generated over $\Z$ by the operators $T_n$ and $T_{n,n}$. 

Fix a finite field extension $F$ of $\Q_p$ and denote $\mathcal O$ its ring of integers. An operator $T\in\mathcal H(U_r,\Delta_r)$ acts on $S_2(U_r,F)$ as follows. Write $T=\coprod_i U\alpha_i$ and define 
\[ (s|T)(g):=\sum_is(g\alpha_i). \] 
Let $\mathfrak h^D_r$ denote the image of $\mathcal H(U_r,\Sigma_r)\otimes_\Z\mathcal O$ in the endomorphism algebra of $S_2(U_r,F)$. As in \S \ref{App1}, for $r\geq1$ let $\mathfrak h_r^{D,\ord}$ denote the product of those local rings of $\mathfrak h^D_r$ where $U_p$ is invertible and define $\mathfrak h_\infty^D$ and $\mathfrak h_\infty^{D,\ord}$ to be the inverse limits over $r\geq1$ of the rings $\mathfrak h^D_r$ and $\mathfrak h_r^{D,\ord}$, respectively. 

For any $a\in\Z'$ choose $\alpha\in\hat\Z\cap\hat\Q^\times$ such that $a\hat\Z=\alpha\hat\Z$ and $\alpha-1\in p^r\hat\Z$, then define $\langle a\rangle=U_r\alpha U_r$ in $\mathcal H(U_r,\Delta_r)$. The map $a\mapsto\langle a\rangle$ is multiplicative and thus extends to a ring homomorphism $\Z[\Z']\rightarrow\mathcal H(\Gamma_r,\Delta_r)$. Since $\Z[\Z']$ embeds naturally in $\tilde\Lambda$, we may form the $\tilde\Lambda$-algebras 
\[ \mathcal H(p^r):=\mathcal H(U_r,\Delta_r)\otimes_{\Z[\Z']}\tilde\Lambda. \]
Now the Hecke pair $(U_r,\Delta_r)$ is weakly compatible (in the sense of \cite[Definition 2.1]{AS}) to $(U_0,\Delta_0)$, hence, as explained in \cite[\S 2]{AS}, there is a canonical surjective $\tilde\Lambda$-algebra homomorphism 
\[ \rho_r:\mathcal H(1)\;\longepi\;\mathcal H(p^r) \] 
for every integer $r\geq1$. 

We define the commutative $\tilde\Lambda$-algebra 
\[ \mathfrak  h^D_{\rm univ}:=\tilde\Lambda\bigl[\,\text{$T_n$ for every $n\geq1$ and $T_{n,n}$ for every $n\geq1$ with $(n,MD)=1$}\,\bigr]\subset\mathcal H(1). \]
The $\tilde\Lambda$-algebra $\mathfrak h^D_{\rm univ}$ acts compatibly on the $\C$-vector spaces $S_2(U_r)$, in the sense that the diagram of $\tilde\Lambda$-algebras
\[ \xymatrix{\mathfrak h^D_{\rm univ}\ar[r]^-{\rho_r}\ar[rd]_-{\rho_{r-1}}&\mathcal H(p^r)\ar@{->>}[d]\\
             &\mathcal H(p^{r-1})} \] 
commutes for all $r\geq1$ (here the vertical arrow is the canonical map arising from the weakly compatibility of the Hecke pairs $(U_r,\Delta_r)$ and $(U_{r-1},\Delta_{r-1})$). For all $r\geq1$ the image of $\mathfrak h^D_{\rm univ}$ in the endomorphism algebra of $S_2(U_r,\C)$ is canonically isomorphic to $\mathfrak h^D_r$, hence we obtain a canonical morphism
\[ h:\mathfrak h^D_{\rm univ}\longrightarrow\mathfrak h_\infty^{D,\ord}. \]

\subsection{Hida families} 

The Jacquet--Langlands correspondence (which, in this case, can be concretely established via the Eichler trace formula) ensures that $\mathfrak h^D_r$ is canonically isomorphic to the quotient of $\mathfrak h_r^1$ acting faithfully on the $\C$-vector space $S_2^\text{$D$-new}(\Gamma_1(MDp^r))$ of cusp forms of weight $2$ and level $\Gamma_1(MDp^r)$ which are new at all the primes dividing $D$. We also fix a (non-canonical) isomorphism $S_2^\text{$D$-new}(\Gamma_1(MDp^r))\simeq S_2(U_r,\C)$. Therefore for all $r\geq1$ there is a canonical projection $\mathfrak h^1_r\rightarrow\mathfrak h^D_r$, which restricts to the ordinary parts, yielding a canonical map $\mathfrak h_\infty^{1,\ord}\rightarrow\mathfrak h_\infty^{D,\ord}$. As above, there is a splitting 
\[ \mathfrak h_\infty^{D,\ord}\otimes_\Lambda\mathcal L\simeq\bigg(\bigoplus_{j\in J}\mathcal F_j\bigg)\oplus\mathcal M \]
where the $\mathcal F_j$ are finite field extensions of $\mathcal L$ and $\mathcal M$ is non-reduced. Since the morphism associated with $f$ factors through $\mathfrak h_\infty^{D,\ord}$, it must factor through some $\mathcal F\in\{\mathcal F_j\}_{j\in J}$ which is canonically isomorphic to $\mathcal K$ (the primitive component through which the morphism associated with $f$ factors). Summing up, we get a commutative diagram
\[ \xymatrix{\mathfrak h_\infty^{1,\ord}\ar[rr]^-{f_\infty}\ar[dr]&&\mathcal R\\
             &\mathfrak h_\infty^{D,\ord}\ar[ur]_-{f_\infty}} \]  
where we write $f_\infty$ also for the factoring map and the unlabeled arrow is the canonical projection considered before. 

\subsection{Hecke action on modular forms} 

For a field $F$, the Hecke action on $S_2(U_r,F)$ has been described above. Now we consider the Hecke action on $S_k(U_r,A)$ for general weights $k$ and $\Q_p$-vector spaces $A$ with a left linear $\GL_2(\Q_p)\cap\M_2(\Z_p)$-action.  

Suppose that $\eta\in\Sigma_r$ and $s\in S_k(U_r,A)$, then set 
\[ (\eta\cdot s)(g):=i_p(\eta_p)s(g\eta). \] 
Observe that
\[ S_k(U_r,F)=\bigl\{s:B^\times\backslash\hat B^\times\rightarrow V_{k-2}(F)\mid\text{$\eta\cdot s=s$ for all $\eta\in U_r$}\bigr\}. \]
Now, if $T=U_r\alpha U_r=\coprod_i \alpha_iU_r$ we define $T\cdot s:=\sum_i\alpha_i\cdot s$. 

\begin{proposition} \label{prop3.3}
For any subfield $F$ of $\C$ containing $F_f$ via the fixed embedding $\bar\Q_p\hookrightarrow\C$ the $\mathfrak h_k^D\otimes_{\mathcal O_f}F$-module $S_k(U_r,F)$ is free of rank $1$ if $k>2$. If $k=2$ then the same is true for the quotient of $S_2(U_r,F)$ by the subspace $S_2^{\rm triv}(U_r,F)$ consisting of those functions factoring through the norm map. 
\end{proposition}

\begin{proof} In order to have uniform notations, for any field $F$ define $\tilde S_k(U_r,F):=S_k(U_r,F)$ if $k>2$ and $\tilde S_2(U_r,F):=S_2(U_r,F)/S_2^{\rm triv}(U_r,F)$. As before, fix a (non-canonical) isomorphism 
\begin{equation} \label{JL-def}
S_k^\text{$D$-new}(\Gamma_1(MDp^r))\simeq\tilde S_k(U_r,\C)
\end{equation}
(\emph{cf.} \cite[Theorem 2]{Bu}). Thanks to \cite[Proposition 3.1]{hida-iwasawa}, there is a canonical isomorphism 
\[ \Hom_\C\bigl(\mathfrak h_k^1\otimes_{\mathcal O_f}\C,\C\bigr)\simeq S_k\bigl(\Gamma_1(MDp^r)\bigr). \] 
Now a homomorphism in the left hand side factors through $\mathfrak h^D_k$ precisely when the corresponding cusp form is new at all the primes dividing $D$, and so, combining this fact with \eqref{JL-def}, we obtain a non-canonical isomorphism 
\[ \Hom_\C\bigl(\mathfrak h_k^D\otimes_{\mathcal O_f}\C,\C\bigr)\simeq\tilde S_k(U_r,\C). \] 
Thanks to \cite[Corollary 6.5]{hida-inv}, $\mathfrak h_r^{1}$ is a Frobenius algebra over $\C$ (actually, over any field $K\subset\C$), hence there is a canonical isomorphism of $\mathfrak h^1_r\otimes_{\mathcal O_f}\C$-modules 
\[ \mathfrak h_r^1\otimes_{\mathcal O_f}\C\simeq\Hom_\C\bigl(\mathfrak h_r^1\otimes_{\mathcal O_f}\C,\C\bigr). \] 
Since $\mathfrak h_r^D\otimes_{\mathcal O_f}\C$ is a direct factor of $\mathfrak h^1_r\otimes_{\mathcal O_f}\C$, one deduces that 
\[ \Hom_\C\bigl(\mathfrak h_r^D\otimes_{\mathcal O_f}\C,\C\bigr)\simeq\mathfrak h^D_r\otimes_{\mathcal O_f}\C \] 
as $\mathfrak h^D_r\otimes\C$-modules. Therefore we get isomorphisms
\[ \bigl(\mathfrak h_r^D\otimes_{\mathcal O_f}F\bigr)\otimes_F\C\simeq\mathfrak h_k^D\otimes_{\mathcal O_f}\C\simeq\tilde S_k(U_r,\C)\simeq\tilde S_k(U_r,F)\otimes_F\C, \] 
and the result follows because $\C$ is fully faithful over $F$. \end{proof}

\subsection{Measure-valued modular forms} \label{W-def}

With notation as in \S \ref{sec3.5}, in our present context define 
\[ \W:=S_2(U_0,\D). \] 
Then $\W$ has a natural action of $\mathfrak h^D_{\rm univ}$ and, since 
\[ \W\simeq\varprojlim_rS_2(U_r,\mathcal O_f), \] 
we can define its ordinary part $\W^\ord$ as in \S \ref{sec3.5}. Moreover, for any $\Lambda$-algebra $R$ put 
\[ \W_R:=\W\otimes_\Lambda R,\qquad\W^\ord_R:=\W^\ord\otimes_\Lambda R. \]
The operator $T_p\in\mathfrak h^D_{\rm univ}$ gives rise to a coset decomposition
\begin{equation} \label{Tp-def}
T_p=\coprod_{a\in\{0,\dots,p-1,\infty\}}U_0g_a
\end{equation}
where $i_\ell(g_\infty)=i_\ell(g_i)=1$ for all $\ell\not=p$ while $i_p(g_\infty)=\smallmat p001$, $i_p(g_i)=\smallmat 1{a_i}0p$ and the $a_i$ are integers forming a complete system of representatives of $\Z/p\Z$.

\subsection{Specialization maps}

As in \S \ref{specialization-subsec}, for an even integer $n\geq0$, a character $\epsilon:\Z_p^\times\rightarrow\bar\Q_p^\times$ factoring through $(\Z/p^r\Z)^\times$ for some integer $r\geq1$ and a finite field extension $L/\Q_p$ containing the values of $\epsilon$ there is a \emph{specialization map}  
\[ \rho_{n,\epsilon}=\rho_{n,\epsilon,L}:\W\longrightarrow S_k(U_r,L) \]
defined by 
\[ \rho_{n,\epsilon}(s)(g)(P):=\int_{\Z_p^\times\times p\Z_p}\epsilon(a)P(x,y)ds(g) \]
for all $g\in\hat B^\times$ and $P\in P_n(L)$. 

Let $\gamma\in B^\times$ and write $i_p(\gamma)=\smallmat abcd$. Suppose that $i_p(\gamma)\in\GL_2(\Q_p)\cap\M_2(\Z_p)$ with $a\in\Z_p^\times$ and $c\equiv0\pmod{p^r}$. Then for $s\in\W$, $g\in\hat B^\times$ and $P\in P_n(L)$ there are equalities 
\[ \begin{split}
   \rho_{n,\epsilon}(\gamma\cdot s)(g)(P)&=\int_\Y\chi_{\Z_p^\times\times p\Z_p}(x,y)\epsilon(a)P(x,y)d(\gamma\cdot s)(g)\\
   &=\int_\Y\chi_{\Z_p^\times\times p\Z_p}(x,y)\epsilon(y)P(x,y)di_p(\gamma)s(g\gamma)\\
   &=\int_\Y\chi_{\Z_p^\times\times p\Z_p}(ax+by,cx+dy)\epsilon(ax+by)P(ax+by,cx+dy)ds(g\gamma)\\
   &=\int_{\Z_p^\times\times p\Z_p}\epsilon(a)\epsilon(x)P(ax+by,cx+dy)ds(g\gamma)\\
   &=\epsilon(a)\rho_{n,\epsilon}(s)(g\gamma)(P|\gamma)=\epsilon(a)\bigl(\gamma\cdot\rho_{n,\epsilon}(s)\bigr)(P).
   \end{split} \] 
The above computation shows that
\[ \rho_{n,\epsilon}(\gamma\cdot s)=\gamma\cdot\rho_{n,\epsilon}(s) \]
for all $\gamma\in\Delta_r$. Note also that $\rho_{n,\epsilon}(T_p s)=U_p\rho_{n,\epsilon}(s)$ because $\smallmat p001(\Z_p^\times\times p\Z_p)\cap\X=\emptyset$. It follows that the map $\rho_{n,\epsilon}$ is $\mathfrak h^D_{\rm univ}$-equivariant.
 
Restricting to the ordinary parts, we also get an $\mathfrak h^D_{\rm univ}$-equivariant map
\[ \rho_{n,\epsilon}^\ord:\W^\ord\longrightarrow S_k(U_r,L)^\ord. \]

\subsection{The Control Theorem} 

For any integer $m\geq1$ and any character $\chi:\Z_p^\times\rightarrow\bar\Q_p^\times$ define the function $\psi_{m,\chi}:\X\rightarrow\Z_p$ by 
\[ \psi_{m,\chi}((x,y)):=\begin{cases}\chi(x)&\text{if $y\in p^m\Z_p$},\\[2mm]0&\text{otherwise}. \end{cases} \]
Note that $\psi_{m,\chi}$ is homogeneous of degree $\chi$ for all integers $m\geq1$. 

The next result is the counterpart of Lemma \ref{prop-kernel}.

\begin{lemma} \label{prop-kernel-def} 
Fix $\kappa\in\mathcal A(\tilde\Lambda)$ of weight $k$ and character $\epsilon$ factoring through $(\Z/p^r\Z)^\times$ for some integer $r\geq1$. Set $n:=k-2$. The map $\rho_{n,\epsilon}^\ord$ induces an injective, $\mathfrak h^D_{\rm univ}$-equivariant homomorphism of $\tilde\Lambda/P_\kappa\tilde\Lambda$-modules
\[ \rho_{n,\epsilon}^\ord:\W^\ord/P_\kappa\W^\ord\;\longmono\;S_k(U_r,F_f)^\ord. \] 
\end{lemma}

\begin{proof} First suppose that $s\in P_\kappa\W^\ord$, so that $s(g)\in P_\kappa\D$. Since the integrand $\epsilon(y)P(x,y)$ appearing in the expression of $\rho_{n,\epsilon}^\ord$ is homogeneous of degree $\kappa$, \cite[Lemma 6.3]{GS} shows that $\rho_{n,\epsilon}(s)(g)(P)=0$ and thus $s\in\ker(\rho_{n,\epsilon}^\ord)$. 

Now we show the opposite inclusion $\ker(\rho_{n,\epsilon}^\ord)\subset P_\kappa\W^\ord$. Fix $s\in \ker(\rho_{n,\epsilon}^\ord)$ and an integer $m\geq1$, then choose $t\in\W^\ord$ such that $T_p^m\cdot t=s$ (this is possible because $T_p$ induces an isomorphism on $\W^\ord$). Write $T_p^m$ as 
\[ T_p^m=\coprod U_0g_{m,i} \] 
where the $g_{m,i}$ are suitable products of $m$ elements, not necessarily distinct, chosen in the set ${\{g_a\}}_{a=0,\dots,p-1,\infty}$ defined in \eqref{Tp-def}.
Therefore we have 
\[ \int_\X \psi_{m,\kappa}(x,y)ds(g)(x,y)=\sum_i\int_\X\psi_{m,\kappa}\big(g_{m,i}(x,y)\big)dt(gg_{m,i})(x,y). \] 
Now $\psi_{m,\kappa}(g_{m,i}(x,y))=0$ unless $g_{m,i}$ is the product of elements of the form $g_i$ with $i\neq\infty$. Furthermore, there is a decomposition 
\[ \Z_p^\times\times p\Z_p=\coprod_{i=0}^{p-1}g_{i}(\Z_p^\times\times p\Z_p), \]
so the above sum is equal to
\[ \sum_{m,i}\int_{g_{m,i}(\Z_p^\times\times p\Z_p)}\kappa(x)dt(gg_{m,i})=\sum_{m,i}\int_{g_{m,i}(\Z_p^\times\times p\Z_p)}\epsilon(x)x^ndt(gg_{m,i})=U_p^m\rho_{n,\epsilon}^\ord(t)(g)(x^n). \] 
Now 
\begin{equation} \label{Up-def-eq}
0=\rho_{n,\epsilon}^\ord(s)=\rho_{n,\epsilon}^\ord(T_p^mt)=U_p^m\rho_{n,\epsilon}^\ord(t). 
\end{equation} 
Since $s\in\W^\ord$, the same is true of $t$ and also of $\rho_{n,\epsilon}^\ord(t)$. Equation \eqref{Up-def-eq} then shows that $t=0$ because $U_p$ acts invertibly on the ordinary submodule, hence we conclude that  
\[ \int_\X \psi_{m,\kappa}(x,y)ds(g)=0 \] 
for all $g\in\hat B^\times$. Finally, from \cite[Lemma 6.3]{GS} it follows that $s(g)\in P_\kappa\D$ for all $g\in\hat B^\times$, and the lemma is proved. \end{proof}

Let $\kappa\in\mathcal A(\mathcal R)$ and define the $F_\kappa$-vector space
\[ \W_\kappa^\ord:=\Big(S_k(U_r,F_\kappa)^\ord\Big)^{f_\kappa}. \] 
Since $f_\kappa$ is either a newform or a $p$-stabilized newform, $\W_\kappa^\ord$ is $1$-dimensional over $F_\kappa$. 

Define $h_\kappa:\mathfrak h^D_{\rm univ}\rightarrow\mathcal R_{P_\kappa}$ to be the composition of $h$ with the localization map $\mathcal R\rightarrow\mathcal R_{P_\kappa}$ at the kernel $P_\kappa$ of $\kappa$. We can consider the $\mathcal R_{P_\kappa}$-submodule $\W_{h_\kappa}^\ord$ of $\W^\ord_{\mathcal R_{P_\kappa}}$ defined by 
\[ \W_{h_\kappa}^\ord:=\bigl(\W^\ord_{\mathcal R_{P_\kappa}}\bigr)^{h_\kappa}. \]

\begin{proposition} \label{propA4-def}
Let $\kappa\in\mathcal A(\mathcal R)$ be of weight $k=k_\kappa$ and character $\epsilon=\epsilon_\kappa$. Set $n:=k-2$. The map $\rho_{n,\epsilon}^\ord$ of Lemma \ref{prop-kernel-def} induces an injective homomorphism of $F_\kappa$-vector spaces
\[ \rho_\kappa:\W_{h_\kappa}^\ord/P_\kappa\W_{h_\kappa}^\ord\;\longmono\;\W_\kappa^\ord. \]
\end{proposition}

\begin{proof} Define $\p_\kappa:=P_\kappa\cap\Lambda$. As in the proof of Proposition \ref{propA4}, Lemma \ref{prop-kernel-def} ensures the existence of an injective homomorphism of $\Lambda_{\p_\kappa}/\p_\kappa\Lambda_{\p_\kappa}$-vector spaces 
\[ \rho_{n,\epsilon}^\ord:\W_{\Lambda_{\p_\kappa}}^\ord/\p_\kappa\W_{\Lambda_{\p_\kappa}}^\ord\;\longmono\;S_k(U_r,F_f)^\ord. \]
Recall that $\mathcal R_{P_\kappa}$ and $\Lambda_{\p_\kappa}$ are normal domains and that $\mathcal R_{P_\kappa}$, being unramified over $\Lambda_{\p_\kappa}$ thanks to \cite[Corollary 1.4]{hida-inv}, is flat over $\Lambda_{\p_\kappa}$. Now 
\[ S(U_r,F_f)\otimes_{F_f}F_\kappa\simeq S_k(U_r,F_\kappa), \] 
from which we easily deduce the result. \end{proof}

Now recall that $\mathcal L$ (respectively, $\mathcal K$) is the fraction field of $\Lambda$ (respectively, $\mathcal R$) and there is a canonical decomposition 
\begin{equation} \label{decomp-def}
\mathfrak h_\infty^{D,\ord}\otimes_\Lambda\mathcal L\simeq\mathcal K\oplus\mathcal N
\end{equation}
where $\mathcal N$ is a direct sum of finitely many fields plus a non-reduced part. As in \S \ref{sec-CT-ind}, consider the composition
\[ h_\mathcal R:{\mathfrak h^D_{\rm univ}}\overset{h}\longrightarrow\mathfrak h_\infty^D\longrightarrow\mathfrak h_\infty^{D,\ord}\overset{f_\infty}\longrightarrow\mathcal R. \]

\begin{proposition} \label{prop-2-def}
The module $\W^\ord_\mathcal K$ is a $1$-dimensional vector space over $\mathcal K$. Furthermore, the action of $\mathfrak h^D_{\rm univ}$ on $\W^\ord_\mathcal K$ factors through $h_\mathcal R$.  
\end{proposition}

\begin{proof} Define the $p$-divisible abelian group 
\[ \mathcal V:=\varinjlim_r S_2(U_r,F_f/\mathcal O_f)^\ord \]
where the direct limit is induced by the maps $U_r\subset U_{r-1}$. The Hecke operators $T_n$, $T_{n,n}$ and the involution $\iota$ act naturally on $\mathcal V$, since the Hecke action is compatible with the restriction maps (see \cite[(2.9 a,b) and (3.5)]{hida}). Define $V$ to be the Pontryagin dual of $\mathcal V$ and note that 
\[ V=\varprojlim_r S_2(U_r,\mathcal O_f)^\ord\simeq\W^\ord. \]
Thanks to \cite[Corollary 10.4]{hida}, we know that $V$ is free of finite rank over $\Lambda$. Using Proposition \ref{prop3.3} in place of Proposition \ref{7.2-hida}, one can proceed as in the proof of Proposition \ref{prop-2} to show that $\W_\mathcal L^\ord$ is free of rank $1$ over $\mathfrak h_\infty^{D,\ord}\otimes_\Lambda\mathcal L$, and we are done thanks to decomposition \eqref{decomp-def}. \end{proof}

Now we can prove the analogue of Theorem \ref{control-thm} in the definite setting.

\begin{theorem} \label{control-thm-def}
For every $\kappa\in\mathcal A(\mathcal R)$ the map $\rho_\kappa$ of Proposition \ref{propA4-def} induces an isomorphism of $1$-dimensional $F_\kappa$-vector spaces
\[ \rho_\kappa:\W^\ord_{h_\kappa}/P_\kappa\W_{h_\kappa}^\ord\overset\simeq\longrightarrow\W^\ord_\kappa. \]
\end{theorem}

\begin{proof} Thanks to Proposition \ref{propA4-def}, we only need to show the inequality
\[ \dim_{F_\kappa}(\W_{h_\kappa}^\ord/P_\kappa\W_{h_\kappa}^\ord)\geq\dim_{F_\kappa}(\W_\kappa^\ord). \] 
Since $\dim_{F_\kappa}(\W_\kappa^\ord)=1$, we are reduced to proving that 
\[ \dim_{F_\kappa}(\W_{h_\kappa}^\ord/P_\kappa\W_{h_\kappa}^\ord)\geq1. \] 
However, Proposition \ref{prop-2-def} shows that the intersection of $\W_{\mathcal R_\kappa}^\ord$ with the $h_\mathcal R$-eigenmodule of $\W^\ord_\mathcal K$ is a free $\mathcal R_{P_\kappa}$-module of rank $1$, which completes the proof. \end{proof}


\begin{thebibliography}{99}

\bibitem{AS} A. Ash, G. Stevens, Cohomology of arithmetic groups and congruences between systems of Hecke eigenvalues, \emph{J. Reine Angew. Math.} {\bf 365} (1986), 192--220.

\bibitem{as} A. Ash, G. Stevens, $p$-adic deformations of cohomology classes of subgroups of $\GL(n,\Z)$, \emph{Collect. Math.} {\bf 48} (1997), no. 1--2, 1--30. 

\bibitem{AS-new} A. Ash, G. Stevens, $p$-adic deformations of arithmetic cohomology, preprint available at \texttt{https://www2.bc.edu/$\sim$ashav/}.

\bibitem{BL} B. Balasubramanyam, M. Longo, $\Lambda$-adic modular symbols over totally real fields, \emph{Comment. Math. Helv.} {\bf 86} (2011), no. 4, 841--865.

\bibitem{BD} M. Bertolini, H. Darmon, Hida families and rational points on elliptic curves, \emph{Invent. Math} {\bf 168} (2007), no. 2, 371--431.

\bibitem{br} G. E. Bredon, \emph{Sheaf theory}, second edition, Graduate Texts in Mathematics {\bf 170}, Springer-Verlag, New York, 1997.  

\bibitem{Bu} K. Buzzard, On $p$-adic families of automorphic forms, in \emph{Modular curves and abelian varieties}, J. Cremona, J.-C. Lario, J. Quer and K. Ribet (eds.), Progress in Mathematics {\bf 224}, Birkh\"auser, Basel, 2004, 23--44.

\bibitem{Chenevier} G. Chenevier, Une correspondance de Jacquet--Langlands $p$-adique, \emph{Duke Math. J.} {\bf 126} (2005), no. 1, 161--194.

\bibitem{DI} F. Diamond, J. Im, Modular forms and modular curves, in \emph{Seminar on Fermat's Last Theorem}, V. Kumar Murty (ed.), CMS Conference Proceedings {\bf 17}, American Mathematical Society, Providence, RI, 1995, 39--133.

\bibitem{Gh} E. Ghate, Ordinary forms and their local Galois representations, in \emph{Algebra and number theory}, R. Tandon (ed.), Hindustan Book Agency, Delhi, 2005, 226--242.

\bibitem{GS} R. Greenberg, G. Stevens, $p$-adic $L$-functions and $p$-adic periods of modular forms, \emph{Invent. Math} {\bf 111} (1993), no. 2, 407--447. 

\bibitem{hida-iwasawa} H. Hida, Iwasawa modules attached to congruences of cusp forms, \emph{Ann. Sci. \'Ecole Norm. Sup. (4)} {\bf 19} (1986), no. 2, 231--273.

\bibitem{hida-inv} H. Hida, Galois representations into ${\rm GL}_2({\bf Z}_p[[X]])$ attached to ordinary cusp forms, \emph{Invent. Math.} {\bf 85} (1986), no. 3, 545--613. 

\bibitem{hida} H. Hida, On $p$-adic Hecke algebras for ${\rm GL}_2$ over totally real fields, \emph{Ann. of Math. (2)} {\bf 128} (1988), no. 2, 295--384.

\bibitem{LRV1} M. Longo, V. Rotger, S. Vigni, On rigid analytic uniformizations of Jacobians of Shimura curves, \emph{Amer. J. Math.}, to appear. 

\bibitem{LV} M. Longo, S. Vigni, Quaternion algebras, Heegner points and the arithmetic of Hida families, \emph{Manuscripta Math.} {\bf 135} (2011), no. 3--4, 273--328.

\bibitem{LV-darmon} M. Longo, S. Vigni, The rationality of quaternionic Darmon points over genus fields of real quadratic fields, arXiv:1105.3721, submitted.

\bibitem{MS} Y. Matsushima, G. Shimura, On the cohomology groups attached to certain vector valued differential forms on the product of the upper half planes, \emph{Ann. of Math. (2)} {\bf 78} (1963), no. 3, 417--449. 

\bibitem{Mi} T. Miyake, \emph{Modular forms}, Springer Monographs in Mathematics, Springer-Verlag, Berlin, 2006.

\bibitem{Sh} G. Shimura, \emph{Introduction to the arithmetic theory of automorphic functions}, Princeton University Press, Princeton, NJ, 1971.

\bibitem{Wiese} G. Wiese, On the faithfulness of parabolic cohomology as a Hecke module over a finite field, \emph{J. Reine Angew. Math.} {\bf 606} (2007), 79--103. 

\end{thebibliography}
\end{document}